\documentclass[pdflatex,sn-mathphys-num]{sn-jnl}

\usepackage{graphicx}%
\usepackage{multirow}%
\usepackage{amsmath,amssymb,amsfonts}%
\usepackage{amsthm}%
\usepackage{mathrsfs}%
\usepackage[title]{appendix}%
\usepackage{xcolor}%
\usepackage{textcomp}%
\usepackage{manyfoot}%
\usepackage{booktabs}%
\usepackage{algorithm}%
\usepackage{algorithmicx}%
\usepackage{algpseudocode}%
\usepackage{listings}%

\usepackage{booktabs}  

\usepackage{amscd,mathtools}
\usepackage{upgreek,bm}
\usepackage{dsfont}
\usepackage{verbatim}
\usepackage{color}
\usepackage{url}
\usepackage{enumerate,enumitem}
\usepackage{epstopdf,epsfig,subfigure}
\usepackage{curve2e}
\usepackage{array}
\usepackage{stackrel}
\usepackage{geometry}
\usepackage{todonotes}
\usepackage[dvipsnames]{xcolor}
\usepackage[normalem]{ulem}

\usepackage{pgfplots}
\pgfplotsset{compat=newest}
\usepgfplotslibrary{groupplots}
\usepackage{tikz}
\definecolor{myblue}{RGB}{70,130,180}
\definecolor{myred}{RGB}{178,34,34}
\definecolor{mygray}{RGB}{140,140,140}

\theoremstyle{thmstyleone}%
\newtheorem{theorem}{Theorem}
\newtheorem{proposition}[theorem]{Proposition}%

\theoremstyle{thmstyletwo}%
\newtheorem{remark}{Remark}%
\newtheorem{assumption}[theorem]{Assumption}%
\newtheorem{lemma}[theorem]{Lemma}%
\newtheorem{corollary}[theorem]{Corollary}%

\theoremstyle{thmstylethree}%
\newtheorem{definition}{Definition}%

\newcommand{\clo}{{\scriptstyle\mathcal{O}}}

\newcommand{\Tr}{\mathop{\rm Tr}\nolimits}
\newcommand{\rest}{\left.\kern-2\nulldelimiterspace\right|_}
\newcommand{\1}{\mathds{1}}

\newcommand{\clC}{{\mathcal C}}

\newcommand{\clL}{{\mathcal L}}

\newcommand{\clN}{{\mathcal N}}

\newcommand{\inp}[3][]{\left\langle #2,#3\right\rangle_{#1}}

\begin{document}

\title[Article Title]{Finite-Dimensional Feedback Stabilization of Nonautonomous Stochastic Parabolic Equations}


\author[1]{\fnm{Behzad} \sur{Azmi}}\email{behzad.azmi@uni-konstanz.de}

\author*[1]{\fnm{Jonas} \spfx{von der} \sur{Heydt}} \email{jonas.von-der-heydt@uni-konstanz.de}

\author[2]{\fnm{S\'ergio} \sur{Rodrigues}}\email{ssi.rodrigues@fct.unl.pt}

\affil[1]{\orgdiv{Dep. Math. Stat.}, \orgname{Univ.
Konstanz}, \orgaddress{\street{Universitätstr. 10}, \city{Konstanz}, \postcode{D-78457}, 
\country{Germany}}}

\affil[2]{\orgdiv{Dep. Math.}, \orgname{NOVA University Lisbon}, \orgaddress{\street{Campus de Caparica},
\city{Caparica}, \postcode{2829-516}, \country{Portugal}}}

\abstract{%
    \unboldmath
    We investigate finite-dimensional feedback stabilization for nonlinear nonautonomous stochastic parabolic equations driven by $Q$-Wiener, covering both additive and multiplicative perturbations. The control is given by a finite linear combination of localized indicator-type actuators whose supports are selected as part of the construction and may have arbitrarily small total measure. The feedback law is constructed by means of oblique projections onto suitable finite-dimensional subspaces. Within the variational Gelfand triple framework, we prove well-posedness of the closed-loop system under standard coercivity, growth, and global Lipschitz assumptions. By appropriately choosing the actuator configuration and feedback strength, we establish exponential mean-square stabilization of the stochastic dynamics and, for pure multiplicative noise, almost-sure stabilization. A fully discrete three-layer implementation complements the theoretical results. Numerical experiments illustrate the influence of number of actuators, noise intensity, and nonlinear effects on the closed-loop stabilization behavior.
}



\maketitle


\section{Introduction}\label{sec:introduction}

In this paper, we study the feedback stabilization of stochastic semilinear parabolic equations by means of finitely many spatially supported actuators. As a model problem, we consider systems of the form
\begin{subequations}\label{eq:sys-X}
\begin{align}
    \mathrm d X
    +\bigl(-\nu\Delta X +&a(t,x)\,X + b(t,x)\cdot\nabla X + f(X)\bigr)\,\mathrm dt
    \notag\\
    &=
    g(t,x,X)\,\mathrm dW_t
    + \sum_{i=1}^{N} u_i(t)\,\1_{\clo_i}(x)\,\mathrm dt,
    \label{eq:sys-X-pde}
    \\
   & \mathfrak T X\rest{\partial\mathcal O}=0,
    \qquad
     X\rest{t=0} = X_0 .
    \label{eq:sys-X-bic}
\end{align}
\end{subequations}
Here, $X=X(t,x)$ denotes the state, a real-valued stochastic process defined for $(t,x)\in[0,\infty)\times\mathcal O$ on a complete filtered probability space $(\Omega,\mathcal F,(\mathcal F_t)_{t\geq0},\mathbb P)$ satisfying the usual conditions. The domain $\mathcal O\subset\mathbb R^d$, $d\geq1$, is bounded, convex, and polygonal or polyhedral, and the boundary operator $\mathfrak T$ represents either homogeneous Dirichlet or homogeneous Neumann boundary conditions. The coefficient $\nu>0$ is fixed, while $a$ and $b$ are deterministic reaction and convection coefficients, and $f$ is a nonlinear reaction term.

The stochastic forcing in~\eqref{eq:sys-X-pde} is driven by a $Q$-Wiener process and allows for both additive and multiplicative noise. The additive component leads to a residual term in the mean-square stabilization estimate, whereas in the purely multiplicative case the closed-loop dynamics decays to zero. The precise assumptions on the covariance operator, the stochastic basis, and the diffusion coefficient are stated in Section~\ref{sec:framework-and-well-posedness}.

The control in~\eqref{eq:sys-X-pde} is finite-dimensional and acts through indicator-type actuators. More precisely, the control term is given by
\begin{equation}\label{eq:intro-control-term}
    \sum_{i=1}^{N} u_i(t)\,\1_{\clo_i}(x)
    \coloneqq
    \sum_{i=1}^{N} [KX(t)]_i\,\1_{\clo_i}(x),
\end{equation}
where $\1_{\clo_i}$ denotes the indicator function of the actuator support $\clo_i\subset\mathcal O$. The sets $\clo_1,\ldots,\clo_N$ are pairwise disjoint and are selected as part of the construction; in particular, their total measure can be prescribed arbitrarily small. The feedback operator $K=K_L^{[\lambda]}$ is constructed explicitly by means of oblique projections. Its construction depends on the actuator supports and on the integer parameter $N$, but not on the coefficients $a$, $b$, $f$, or $g$.

Although \eqref{eq:sys-X} serves as the guiding model problem throughout the introduction, the theoretical analysis is carried out for the more general abstract closed-loop equation~\eqref{eq:abstract-spde} introduced in Section~\ref{sec:framework-and-well-posedness}. The main well-posedness and stabilization results stated below apply at that level of generality.

The main result shows that, under the standing assumptions, one can construct an oblique-projection feedback that stabilizes the closed-loop system at any prescribed exponential rate. More precisely, for every $\mu>0$, there exist a suitable actuator configuration, a feedback parameter $\lambda>0$, and the corresponding feedback operator $K=K_L^{[\lambda]}$ such that
\begin{equation}\label{eq:intro-goal-feed-Exp}
    \mathbb E\|X(t)\|_{L^2(\mathcal O)}^2
    \leq
    e^{-2\mu t}\,\mathbb E\|X_0\|_{L^2(\mathcal O)}^2
    + C_{\mathrm{add}},
    \qquad t\geq0.
\end{equation}
Here $C_{\mathrm{add}}\geq0$ depends on the additive component of the noise. Thus, in the presence of additive noise, the feedback stabilizes the system up to a noise-dependent residual bound. If the additive component is absent, then $C_{\mathrm{add}}=0$, and the estimate yields exponential mean-square decay to zero. In the purely multiplicative case, we also prove the pathwise estimate
\begin{equation}\label{eq:intro-goal-feed-P}
    \limsup_{t\to\infty}
    \frac{1}{t}\log \|X(t)\|_{L^2(\mathcal O)}^2
    \leq -2\mu,
    \qquad \mathbb P\text{-a.s.}
\end{equation}
Hence the closed-loop system is exponentially stable almost surely. The precise statement is given in Theorem~\ref{thm:main} in Section~\ref{subsec:exponential-stability}.

The paper also contains a numerical study of the proposed feedback law. In a concrete two-dimensional setting, we describe a fully discrete three-layer implementation of the closed-loop system and use it to investigate the stabilization behavior under different actuator configurations, noise structures, noise intensities, and nonlinear dynamics. We also derive a mean-square Karhunen--Lo\`eve truncation error estimate for the covariance operators $Q = (-\Delta_{\mathrm{Neu}} + \ell^{-2})^{-\alpha}$, which provides a rigorous justification for the approximation of the stochastic forcing used in the simulations. The simulations illustrate the effectiveness of the feedback, confirm numerically the predicted truncation rate, and indicate that the associated state error decays even faster in the reported examples.

\paragraph{Related work.}
The stabilization of stochastic parabolic equations has been studied in several directions, but finite-dimensional actuator-based stabilization remains much less developed than in the deterministic setting. Early stabilization results for stochastic heat equations with multiplicative noise were obtained in~\cite{barbu2002note}. A particularly important reference for the present work is~\cite{Barbu13}, where internal stabilization of stochastic parabolic equations with linearly multiplicative Gaussian noise was established by means of a linear feedback localized in an interior subdomain.

By contrast, finite-dimensional stabilization is much better understood for deterministic systems. Internal exponential stabilization of the three-dimensional Navier--Stokes equations to a nonstationary solution was proved in~\cite{Barbu2011internal} by means of truncated observability properties. Related finite-dimensional stabilization results for deterministic parabolic systems were obtained in~\cite{Kroener2015,Azmi2022}. The feedback construction used here is inspired in particular by the deterministic oblique-projection framework developed in~\cite{KunRod19-cocv,Rod21-aut}; see also~\cite{KunRodWal21,azmi2019oblique} for related semilinear parabolic and semilinear damped wave-like equations, and~\cite{azmi2024rhc} for a random linear parabolic setting.

For stochastic parabolic equations, the corresponding finite-dimensional theory is less direct. As noted in~\cite{Barbu13}, an analogue based on finitely many actuators would require stochastic counterparts of truncated observability inequalities, whose derivation is substantially more delicate than in the deterministic case. More generally, the controllability theory for stochastic parabolic equations does not parallel the deterministic theory in a straightforward way; see, for instance,~\cite{Lu2022Zhang}. The present paper contributes to this line of work by constructing an explicit actuator-based feedback for nonautonomous stochastic parabolic equations driven by a $Q$-Wiener process.

Compared with~\cite{Barbu13}, the feedback constructed here is finite-dimensional and obtained by an oblique-projection argument. The actuator supports are selected as part of the construction, and their total measure can be prescribed arbitrarily small. In addition, the present analysis covers both additive and multiplicative noise and yields a prescribed exponential decay rate within the stated assumptions. A recent preprint~\cite{hernandez2026stability} studies finite-dimensional feedback stabilization for a class of stochastic heat equations with purely multiplicative noise driven by a standard one-dimensional Brownian motion. That work is close in spirit to the present one, but the stochastic setting and the feedback construction are substantially different.

Feedback stabilization is also closely related to stochastic linear-quadratic control and Riccati theory. In that framework, stabilizing feedbacks are typically obtained from operator-valued Riccati equations or backward stochastic Riccati equations; see, for instance,~\cite{GuatteriTessitore2015,LuZhang2024,LuWang2023} and the references therein. For recent work on numerical discretizations of the associated Riccati equations, see~\cite{ProhlWang2024}. In contrast, the feedback in the present paper is not derived from an optimal control problem, but is constructed explicitly by means of oblique projections. In the linear case, the result provides stabilizability by finitely many indicator-type actuators, a property that is relevant in infinite-horizon Riccati theory for the existence of stabilizing Riccati solutions with such controls.

For the well-posedness analysis, we use the variational formulation and existence theory for stochastic partial differential equations from~\cite{Liu06}; see also the classical works~\cite{pardoux1975,krylov_rozovskii1981} and the systematic presentation in~\cite{liu_roeckner2015}. This framework allows us to formulate the feedback-controlled system in variational form and to establish well-posedness of the abstract closed-loop equation in Section~\ref{sec:framework-and-well-posedness}. The subsequent stabilization analysis must then control both the unstable deterministic modes and the additional terms generated by the It\^o correction and the quadratic variation of the noise.

\paragraph{Contributions.}
The contributions of the present work are as follows.
\begin{itemize}
    \item We construct an explicit oblique-projection feedback for nonautonomous stochastic parabolic equations driven by a $Q$-Wiener process and allowing for both additive and multiplicative noise. The resulting control acts through finitely many indicator-type actuators whose supports are selected as part of the construction and whose total measure can be arbitrarily small.\\We also indicate how the same oblique-projection mechanism can be interpreted as an observer injection for stochastic state estimation, or continuous data assimilation, from localized noisy measurements (see Sect.~\ref{subsec:observer-interpretation}).
    \item We establish well-posedness of the abstract closed-loop system in the variational framework and prove exponential stabilization in the mean-square sense. In the purely multiplicative case, we also prove pathwise exponential stabilization.
    \item On the computational side, in a concrete two-dimensional setting, we present a fully discrete three-layer implementation of the closed-loop system and formulate the corresponding algorithm. We also derive a mean-square Karhunen--Lo\`eve truncation error estimate for the covariance operators $Q = (-\Delta_{\mathrm{Neu}} + \ell^{-2})^{-\alpha}$.
    \item We complement the theoretical analysis by numerical simulations that illustrate the performance of the feedback for different numbers of actuators, noise structures, noise intensities, and nonlinear dynamics. The simulations confirm numerically the predicted truncation rate and indicate that the associated state error decays even faster in the reported examples.
\end{itemize}

\paragraph{Organization.}
In Section~\ref{sec:framework-and-well-posedness} we introduce the variational Gelfand triple framework, formulate the abstract closed-loop equation~\eqref{eq:abstract-spde}, state the standing assumptions, verify them for the model system~\eqref{eq:sys-X}, and establish global well-posedness together with the It\^o formulas used in the stability analysis. In Section~\ref{sec:feed-stab} we construct the oblique-projection feedback operator $K_L^{[\lambda]}$ and prove the main stabilization result. Section~\ref{sec:numerical-scheme} introduces the fully discrete numerical scheme and the Karhunen--Lo\`eve truncation analysis. Numerical experiments are presented in Section~\ref{sec:simulations}.

\section{Framework and Well-Posedness}\label{sec:framework-and-well-posedness}

This section introduces the variational framework used throughout the paper. After specifying the functional setting and the noise, we cast \eqref{eq:sys-X} in abstract form, verify the required assumptions for the concrete equation, establish global well-posedness, and derive the It\^o formulas needed in Section~\ref{sec:feed-stab}.

\subsection{Functional Setting}\label{subsec:functional-setting}

We begin with the concrete class of stochastic parabolic equations introduced in Section~\ref{sec:introduction}. Let $\mathcal O\subset\mathbb R^d$, $d\in\mathbb N_+$, be a bounded convex polygonal or polyhedral domain with boundary $\partial\mathcal O$, and let $\mathfrak T$ denote either the homogeneous Dirichlet trace operator $\Tr$ or the homogeneous Neumann trace operator $\partial_n$. We work throughout in the Gelfand triple
\begin{equation}\label{eq:gelfand-triple}
    V \xhookrightarrow{d,c} H \xhookrightarrow{d} V',
\end{equation}
where $H:=L^2(\mathcal O)$ is the pivot space, identified with its dual $H'\equiv H$, and
\[
    V:=H^1_{\mathfrak T}(\mathcal O)
    :=
    \begin{cases}
        H^1_0(\mathcal O), & \mathfrak T=\Tr,\\
        H^1(\mathcal O),   & \mathfrak T=\partial_n.
    \end{cases}
\]
We endow $V$ with the norm $\|u\|_V^2:=\|u\|_H^2+\|\nabla u\|_H^2$. Since $\mathcal O$ is bounded and Lipschitz, the embedding $V\hookrightarrow H$ is dense and compact, and the embedding $H\hookrightarrow V'$ is the canonical dense injection. Moreover,
\begin{equation}\label{eq:embed-contractive}
    \|u\|_H\leq \|u\|_V \quad \forall\,u\in V,
    \qquad
    \|h\|_{V'}\leq \|h\|_H \quad \forall\,h\in H.
\end{equation}
The duality pairing $\langle \cdot,\cdot\rangle_{V',V}$ extends the $H$-inner product, that is,
\[
    \langle h,v\rangle_{V',V}=\langle h,v\rangle_H
    \qquad \forall\,h\in H,\ \forall\,v\in V.
\]
For the model~\eqref{eq:sys-X}, we assume that $\nu>0$ is constant and
that the reaction and convection coefficients $a$ and $b$ are
deterministic and essentially bounded, with
\begin{equation}\label{eq:bdd-ab}
    a\in L^\infty((0,\infty)\times\mathcal O),
    \qquad
    b\in L^\infty((0,\infty)\times\mathcal O)^d .
\end{equation}
Finally, the nonlinearity $f\colon H\to H$ is assumed to be globally Lipschitz continuous. We further consider a   stochastic diffusion coefficient of the form $g(t,x,r)=\xi(t,x)+z(r)$, where $z\colon\mathbb R\to\mathbb R$ is globally Lipschitz continuous and $\xi$ is deterministic with
\begin{equation}\label{eq:bdd-xi}
    \xi \in L^\infty((0,\infty)\times\mathcal O).
\end{equation}

\subsection{Probability Space and Noise}\label{subsec:probability-space-noise}

We fix a complete filtered probability space $(\Omega,\mathcal F,(\mathcal F_t)_{t\geq0},\mathbb P)$ satisfying the usual conditions. The noise space is $U:=L^2(\mathcal O)=H$. Let $Q\in\mathcal L(U)$ be a bounded, nonnegative, self-adjoint, trace-class operator, and let $(e_k)_{k\geq1}$ be an orthonormal basis of $U$ consisting of eigenvectors of $Q$ with corresponding eigenvalues $(\lambda_k)_{k\geq1}$, $\lambda_k\geq0$, that is,
\[
    Qe_k=\lambda_k e_k,
    \qquad k\geq1.
\]
Let $(\beta_k)_{k\geq1}$ be a sequence of independent standard Brownian motions on
$(\Omega,\mathcal F,(\mathcal F_t)_{t\geq0},\mathbb P)$. The associated $Q$-Wiener process is given by
\begin{equation}\label{eq:W-repr}
    W(t)=\sum_{k=1}^\infty \sqrt{\lambda_k}\,\beta_k(t)\,e_k.
\end{equation}
Since $Q$ is trace class, the series in \eqref{eq:W-repr} converges in $U$. We also introduce the Hilbert space
\begin{equation}\label{eq:U0}
    U_0:=Q^{1/2}(U),
    \qquad
    \langle u,v\rangle_{U_0}
    :=
    \langle Q^{-1/2}u,Q^{-1/2}v\rangle_U,
\end{equation}
where $Q^{-1/2}$ denotes the standard pseudoinverse of $Q^{1/2}$. We write $\mathcal L_2^0 := \mathcal L_2(U_0,H)$ for the space of Hilbert--Schmidt operators from $U_0$ into $H$, equipped with its standard Hilbert--Schmidt norm $\|\cdot\|_{\mathcal L_2^0}$. In the abstract framework below, the diffusion coefficients are required to take values in $\mathcal L_2^0$. For multiplication operators, this requirement is naturally encoded by square integrability with respect to the weighted measure $q(x)\,\mathrm dx$, where
\[
    q(x):=\sum_{k=1}^\infty \lambda_k e_k(x)^2.
\]
For the concrete diffusion coefficients considered in this paper, it is convenient to impose the stronger condition
\begin{equation}\label{eq:Q-cond}
    q\in L^\infty(\mathcal O).
\end{equation}
Combined with the standing assumptions on $z$ and $\xi$, this ensures that the operators
$v \mapsto z(u)\,v$ and $v \mapsto \xi(t,\cdot)\,v$
belong to $\mathcal L_2^0$ and satisfy the Lipschitz and growth conditions in Assumptions~\ref{ass:Z} and~\ref{ass:Xi}; see Remark~\ref{rem:verification}. The condition~\eqref{eq:Q-cond} is not imposed at the level of the abstract equation and is not needed for the abstract well-posedness and stabilization results proved below. In particular, \eqref{eq:Q-cond} holds whenever
\begin{equation}\label{eq:sufficient-Q-cond}
    \sum_{k=1}^\infty \lambda_k \|e_k\|_{L^\infty(\mathcal O)}^2<\infty.
\end{equation}


\subsection{Abstract formulation and standing assumptions}
\label{subsec:abstract-equation-assumptions}

We now formulate~\eqref{eq:sys-X} in an abstract variational setting and state the assumptions on the coefficients used in the well-posedness analysis. At the abstract level, the covariance operator $Q$ is assumed to be nonnegative, self-adjoint, and trace class. No pointwise condition on the associated weight is needed here; such a condition is only used later when the abstract assumptions are verified for the concrete multiplication-type noise.

We denote by $A\colon V\to V'$ the principal coercive part, by $A_{\mathrm{rc}}(t)$ the linear reaction--convection contribution, and by $F(t,\cdot)$ the nonlinear reaction term. Thus, in the concrete model, one should think of
\[
    Av=-\nu\Delta v,
    \qquad
    A_{\mathrm{rc}}(t)v=a(t,\cdot)v+b(t,\cdot)\cdot\nabla v,
    \qquad
    F(t,v)=f(v).
\]
The abstract closed-loop equation reads
\begin{equation}\label{eq:abstract-spde}
    \begin{split}
    \mathrm dX
    + \bigl(AX+A_{\mathrm{rc}}(t)X+F(t,X)\bigr)\,\mathrm dt
    &=
    \bigl(Z(t,X)+\Xi(t)\bigr)\,\mathrm dW
    + U_N^\diamond KX\,\mathrm dt,
    \\
    X\rest{t=0} &= X_0 .
    \end{split}
\end{equation}
The precise meaning of~\eqref{eq:abstract-spde} as an identity in $V'$
is given in Section~\ref{subsec:well-posedness}. The input operator $U_N^\diamond\colon\mathbb R^N\to H$ is defined by
\begin{equation}\label{eq:UN-diamond}
    U_N^\diamond v
    :=
    \sum_{i=1}^N v_i\,\1_{\clo_i},
    \qquad
    v=(v_1,\dots,v_N)^\top\in\mathbb R^N .
\end{equation}
Since $K\in\mathcal L(H,\mathbb R^N)$, we have
$U_N^\diamond K\in\mathcal L(H)$. Hence the feedback term
$U_N^\diamond KX$ is a bounded linear perturbation of the drift.

We impose the following standing assumptions on the coefficients.

\begin{assumption}[Standing coefficient assumptions]\label{ass:standing}
All assertions involving the time variable are understood to hold for a.e. $t\geq0$, and all constants below are independent of $t$.

\begin{enumerate}[label=\textup{(A\arabic*)}]

    \item\label{ass:A}
    The operator $A\in\mathcal L(V,V')$ is symmetric and positive definite with respect to the $V$-norm: there exists $\alpha_A>0$
    such that
    \begin{equation}\label{eq:A-positive}
        \langle Ay,y\rangle_{V',V}
        \geq
        \alpha_A\|y\|_V^2
        \qquad\forall\,y\in V .
    \end{equation}
    In this case, $(y,z)\mapsto\langle Ay,z\rangle_{V',V}$ is a complete scalar product on~$V.$
    \item\label{ass:Arc}
    Let us denote
    \begin{equation}\label{eq:X-space}
        \mathfrak X
        :=
        \mathcal L(V,H)+\mathcal L(H,V').
    \end{equation}
    Then, for a.e.\ $t\geq0$, we have $A_{\mathrm{rc}}(t)\in\mathfrak X$.
    Moreover, $A_{\mathrm{rc}}$ is essentially bounded in $\mathfrak X$, i.e.,
    \begin{equation}\label{eq:Arc-X-bound}
        \|A_{\mathrm{rc}}\|_{L^\infty([0,\infty);\mathfrak X)}
        =:
        C_{\mathrm{rc}}
        <\infty.
    \end{equation}
    By $A_{\mathrm{rc}}(t) \in \mathcal L(H,V') + \mathcal L(V,H)$, we mean that $A_{\mathrm{rc}}(t) = A_1(t) + A_2(t)$ with $A_1(t) \in \mathcal L(H,V')$ and $A_2(t) \in \mathcal L(V,H)$.
    \item\label{ass:N}
    The map $(t,y)\mapsto F(t,y)$ is $(\mathcal B([0,\infty))\otimes\mathcal B(V))$-measurable as a map $[0,\infty)\times V\to H$. For a.e.\ $t\geq0$, the map $F(t,\cdot)\colon V\to H$ is continuous and satisfies $F(t,0)=0$. Moreover, $F$ is uniformly Lipschitz with respect to the $H$-norm: there exists $L_{F}>0$ such that
    \begin{equation}\label{eq:N-lip}
        \|F(t,y)-F(t,z)\|_H
        \leq
        L_{F}\|y-z\|_H
        \qquad
        \forall\,y,z\in V .
    \end{equation}

    \item\label{ass:Z}
    With $\mathcal L_2^0:=\mathcal L_2(U_0,H)$, the map
    $Z\colon[0,\infty)\times H\to\mathcal L_2^0$ is
    $(\mathcal B([0,\infty))\otimes\mathcal B(H))$-measurable. Moreover,
    there exists $L_Z>0$ such that, for all $u,v\in H$,
    \begin{align}
        \|Z(t,u)-Z(t,v)\|_{\mathcal L_2^0}^2
        &\leq
        L_Z^2\|u-v\|_H^2,
        \label{eq:Z-lip}
        \\
        \|Z(t,u)\|_{\mathcal L_2^0}^2
        &\leq
        L_Z^2\bigl(1+\|u\|_H^2\bigr).
        \label{eq:Z-growth}
    \end{align}

    \item\label{ass:Xi}
    The additive component
    $\Xi\colon[0,\infty)\to\mathcal L_2^0$ is
    $\mathcal B([0,\infty))$-measurable and satisfies
    \begin{equation}\label{eq:Xi-bdd}
        \Xi_\infty
        :=
        \|\Xi\|_{L^\infty([0,\infty);\mathcal L_2^0)}
        <\infty .
    \end{equation}

\end{enumerate}
\end{assumption}

The following consequence of Assumption~\ref{ass:standing},
items~\ref{ass:Z}--\ref{ass:Xi}, will be used repeatedly. For a.e.\
$t\geq0$ and all $v\in H$,
\begin{equation}\label{eq:trace-bound}
    \|Z(t,v)+\Xi(t)\|_{\mathcal L_2^0}^2
    \leq
    \mathfrak w_1\|v\|_H^2+\mathfrak w_0,
\end{equation}
where
\begin{equation}\label{eq:w-def}
    \mathfrak w_1:=2L_Z^2,
    \qquad
    \mathfrak w_0:=2L_Z^2+2\Xi_\infty^2 .
\end{equation}

\begin{remark}\label{rem:X-space}
The use of the sum space $\mathfrak X=\mathcal L(V,H)+\mathcal L(H,V')$ in \ref{ass:Arc} is only a convenient way to cover both types of lower-order terms in one notation. In particular, if $A_{\mathrm{rc}}(t)\in\mathcal L(H,V')$, then one may use the decomposition $A_{\mathrm{rc}}(t)=A_{\mathrm{rc}}(t)+0$, whereas if $A_{\mathrm{rc}}(t)\in\mathcal L(V,H)$, then one may use the decomposition $A_{\mathrm{rc}}(t)=0+A_{\mathrm{rc}}(t)$. Thus each of the two cases is included as a special case of $A_{\mathrm{rc}}(t)\in\mathfrak X$. In the latter case, for example, we have 
\[
    \|A_{\mathrm{rc}}(t)\|_{\mathfrak X}
    \leq
    \|A_{\mathrm{rc}}(t)\|_{\mathcal L(V,H)}.
\]
\end{remark}

\begin{remark}[Verification for the concrete system]\label{rem:verification}
For the concrete equation~\eqref{eq:sys-X}, the assumptions above are satisfied under the standing hypotheses on the coefficients. In the Dirichlet case, the Laplace operator satisfies Assumption~\ref{ass:A}. In the Neumann case, one may replace the principal part by $A_\rho:=-\nu\Delta+\rho I$ with $\rho>0$. The compensating term $-\rho I$ is then included in $A_{\mathrm{rc}}(t)$. Thus Assumption~\ref{ass:A} does not impose a genuine restriction on the concrete equation.

The reaction--convection operator $A_{\mathrm{rc}}(t)v=a(t,\cdot)v+b(t,\cdot)\cdot\nabla v$ belongs to $\mathfrak X=\mathcal L(V,H)+\mathcal L(H,V')$. Indeed, the reaction part $v\mapsto a(t,\cdot)v$ is bounded from $H$ to $H$, and hence from $H$ to $V'$, while the convection part $v\mapsto b(t,\cdot)\cdot\nabla v$ is bounded from $V$ to $H$. The uniform boundedness of $a$ and $b$ gives~\eqref{eq:Arc-X-bound}. If the term $-\rho I$ is added in the Neumann case as above, it is treated as part of the reaction contribution.

The nonlinear assumption follows from the global Lipschitz continuity of $f$ and the condition $f(0)=0$, with $F(t,v)=f(v)$ and  $L_{F}=L_f$. Finally, the diffusion assumptions follow from the global Lipschitz continuity of $z$, the boundedness of $\xi$, and the Hilbert--Schmidt estimates for the corresponding multiplication operators. In particular, if $q\in L^\infty(\mathcal O)$, then
\[
    \|Z(t,u)-Z(t,v)\|_{\mathcal L_2^0}^2
    \leq
    \|q\|_{L^\infty(\mathcal O)}L_z^2\|u-v\|_H^2,
\]
and analogous estimates yield the linear growth of $Z$ and the boundedness of $\Xi$.
\end{remark}

\subsection{Well-posedness}
\label{subsec:well-posedness}

We establish well-posedness of~\eqref{eq:abstract-spde} within the variational framework of Liu~\cite[Sec.~1.3.1]{Liu06}. The central object is the stochastic evolution equation
\begin{equation}\label{eq:Liu-integral}
    X_t
    =
    X_0
    +
    \int_0^t \mathbf A(s,X_s)\,\mathrm ds
    +
    \int_0^t \mathbf B(s,X_s)\,\mathrm dW_s ,
\end{equation}
where $\mathbf A(t,\cdot)\colon V\to V'$ and $\mathbf B(t,\cdot)\colon V\to \mathcal L_2^0$ denote the drift and diffusion coefficients, respectively. To rewrite  \eqref{eq:abstract-spde} in the form~\eqref{eq:Liu-integral}, we move the drift terms to the right-hand side and set
\begin{equation}\label{eq:AB-def}
    \mathbf A(t,v)
    :=
    -Av-A_{\mathrm{rc}}(t)v-F(t,v)+U_N^\diamond Kv,
    \qquad
    \mathbf B(t,v)
    :=
    Z(t,v)+\Xi(t).
\end{equation}

In order to apply \cite[Thm.~1.3.1]{Liu06} to the controlled equation~\eqref{eq:abstract-spde}, we verify that the coefficients $\mathbf A$ and $\mathbf B$ defined in~\eqref{eq:AB-def} satisfy the following conditions.

\begin{assumption}[Liu's framework conditions]\label{ass:Liu}
There exist constants $\alpha>0$, $c>0$, $L_{\mathbf B}>0$, and $\eta,\gamma\in\mathbb R$ such that, for all $u,v,w\in V$ and a.e. $t\geq0$, the following conditions hold.
\begin{enumerate}[label=\textup{(\alph*)}]

    \item\label{ass:Liu-coercivity}
    \emph{Coercivity.}
    \begin{equation}\label{eq:Liu-coercivity}
        2\langle \mathbf A(t,v),v\rangle_{V',V}
        +
        \|\mathbf B(t,v)\|_{\mathcal L_2^0}^2
        \leq
        -\alpha\|v\|_V^2+\eta\|v\|_H^2+\gamma .
    \end{equation}

    \item\label{ass:Liu-growth}
    \emph{Growth.}
    \begin{equation}\label{eq:Liu-growth}
        \|\mathbf A(t,v)\|_{V'}
        \leq
        c\bigl(1+\|v\|_V\bigr).
    \end{equation}

    \item\label{ass:Liu-monotonicity}
    \emph{Monotonicity.}
    \begin{equation}\label{eq:Liu-monotonicity}
        2\langle \mathbf A(t,u)-\mathbf A(t,v),u-v\rangle_{V',V}
        \leq
        \eta\|u-v\|_H^2
        -
        \|\mathbf B(t,u)-\mathbf B(t,v)\|_{\mathcal L_2^0}^2 .
    \end{equation}

    \item\label{ass:Liu-hemicontinuity}
    \emph{Hemicontinuity.}
    The map
    \begin{equation}\label{eq:Liu-hemicontinuity}
        \mathbb R\ni\phi
        \mapsto
        \langle \mathbf A(t,u+\phi v),w\rangle_{V',V}
        \in\mathbb R
    \end{equation}
    is continuous.

    \item\label{ass:Liu-Lipschitz}
    \emph{Lipschitz condition on $\mathbf B$.}
    \begin{equation}\label{eq:Liu-Lipschitz}
        \|\mathbf B(t,u)-\mathbf B(t,v)\|_{\mathcal L_2^0}
        \leq
        L_{\mathbf B}\|u-v\|_V .
    \end{equation}

\end{enumerate}
\end{assumption}

\begin{proposition}[Verification of Assumption~\ref{ass:Liu}]
\label{prop:Liu-verification}
Suppose that Assumption~\ref{ass:standing} holds and that $\mathbf A$ and $\mathbf B$ are defined by~\eqref{eq:AB-def}. Then $\mathbf A$ and $\mathbf B$ satisfy Assumption~\ref{ass:Liu}. More precisely, let $\kappa_K:=\|U_N^\diamond K\|_{\mathcal L(H)}$. For any $\varepsilon\in(0,2\alpha_A)$, one may choose
\begin{equation}\label{eq:Liu-constants}
\begin{gathered}
    \alpha:=2\alpha_A-\varepsilon,\qquad
    c:=\|A\|_{\mathcal L(V,V')}+C_{\mathrm{rc}}+L_{F}+\kappa_K,
    \qquad
    L_{\mathbf B}:=L_Z,\\
    \eta:=\varepsilon^{-1}C_{\mathrm{rc}}^2
    +2L_{F}+2\kappa_K+\mathfrak w_1,
    \qquad
    \gamma:=\mathfrak w_0 .
\end{gathered}
\end{equation}
\end{proposition}

\begin{proof}
Fix $t\geq0$ such that Assumption~\ref{ass:standing} holds at time $t$. We verify conditions~\ref{ass:Liu-coercivity}--\ref{ass:Liu-Lipschitz}.

\smallskip
\noindent\textbf{\ref{ass:Liu-coercivity}~Coercivity.}
Let $v\in V$. By~\eqref{eq:AB-def}, we can write 
\[
\begin{aligned}
    &2\langle \mathbf A(t,v),v\rangle_{V',V}
    +
    \|\mathbf B(t,v)\|_{\mathcal L_2^0}^2
    =
    -2\langle Av,v\rangle_{V',V}
    -2\langle A_{\mathrm{rc}}(t)v,v\rangle_{V',V}
    \\
    &\quad
    -2\langle F(t,v),v\rangle_H
    +2\langle U_N^\diamond Kv,v\rangle_H +\|Z(t,v)+\Xi(t)\|_{\mathcal L_2^0}^2 .
\end{aligned}
\]
Using~\eqref{eq:A-positive}, the fact that
\[
    |\langle A_{\mathrm{rc}}(t)v,v\rangle_{V',V}|
    \leq
    C_{\mathrm{rc}}\|v\|_V\|v\|_H,
\]
and Young's inequality, we have for every $\varepsilon>0$ that 
\begin{equation}
\label{eq:cor_est1}
    -2\langle A_{\mathrm{rc}}(t)v,v\rangle_{V',V}
    \leq
    \varepsilon\|v\|_V^2
    +
    \varepsilon^{-1}C_{\mathrm{rc}}^2\|v\|_H^2.
\end{equation}
Moreover, by $F(t,0)=0$ and~\eqref{eq:N-lip}, we obtain  $\|F(t,v)\|_H  \leq  L_{F}\|v\|_H$,  and hence
\[
    -2\langle F(t,v),v\rangle_H
    \leq
    2L_{F}\|v\|_H^2.
\]
The feedback term is estimated by
\begin{equation}
\label{eq:cor_est2}
    2\langle U_N^\diamond Kv,v\rangle_H
    \leq
    2\kappa_K\|v\|_H^2 .
\end{equation}
Finally, by~\eqref{eq:trace-bound}, together with the above estimates, we can infer that
\[
\begin{aligned}
    2\langle \mathbf A(t,v),v\rangle_{V',V}
    +
    \|\mathbf B(t,v)\|_{\mathcal L_2^0}^2
    &\leq
    -(2\alpha_A-\varepsilon)\|v\|_V^2
    \\
    &\quad
    +
    \bigl(
        \varepsilon^{-1}C_{\mathrm{rc}}^2
        +2L_{F}
        +2\kappa_K
        +\mathfrak w_1
    \bigr)
    \|v\|_H^2
    +
    \mathfrak w_0 .
\end{aligned}
\]
Thus~\eqref{eq:Liu-coercivity} holds with the constants
in~\eqref{eq:Liu-constants}, provided
$\varepsilon\in(0,2\alpha_A)$.

\smallskip
\noindent\textbf{\ref{ass:Liu-growth}~Growth.}
Let $v\in V$. Using~\eqref{eq:AB-def}, Assumption~\ref{ass:standing}, and \eqref{eq:embed-contractive}, we have  we obtain
\[
\begin{aligned}
   & \|\mathbf A(t,v)\|_{V'}
    \leq
    \|Av\|_{V'}
    +
    \|A_{\mathrm{rc}}(t)v\|_{V'}
    +
    \|F(t,v)\|_{V'}
    +
    \|U_N^\diamond Kv\|_{V'}
    \\
    &\leq
    \|A\|_{\mathcal L(V,V')}\|v\|_V
    +
    C_{\mathrm{rc}}\|v\|_V
    +
    L_{F}\|v\|_H
    +
    \kappa_K\|v\|_H
    \\
    &\leq
    \bigl(
        \|A\|_{\mathcal L(V,V')}
        +C_{\mathrm{rc}}
        +L_{F}
        +\kappa_K
    \bigr)
    \|v\|_V
    \leq
    c\bigl(1+\|v\|_V\bigr).
\end{aligned}
\]
Here we have used the canonical embeddings $H\hookrightarrow V'$ and $V\hookrightarrow H$.

\smallskip
\noindent\textbf{\ref{ass:Liu-monotonicity}~Monotonicity.}
Let $u,v\in V$ and set $w:=u-v$. By~\eqref{eq:AB-def},
\begin{equation}
\label{eq:mon_esti1}
\begin{aligned}
    2\langle \mathbf A(t,u)-\mathbf A(t,v),w\rangle_{V',V}
    &=
    -2\langle Aw,w\rangle_{V',V}
    -2\langle A_{\mathrm{rc}}(t)w,w\rangle_{V',V}
    \\
    &\quad
    -2\langle F(t,u)-F(t,v),w\rangle_H
    +2\langle U_N^\diamond Kw,w\rangle_H .
\end{aligned}
\end{equation}
The Lipschitz continuity of $F$ yields
\[
    -2\langle F(t,u)-F(t,v),w\rangle_H
    \leq
    2L_{F}\|w\|_H^2,
\]
Together with using estimates  \eqref{eq:A-positive}, \eqref{eq:cor_est1}, and \eqref{eq:cor_est2} for $w$, in \eqref{eq:mon_esti1}, we obtain
\begin{equation}
\label{eq:mon_esti2}
\begin{aligned}
    2\langle \mathbf A(t,u)-\mathbf A(t,v),w\rangle_{V',V}
    &\leq
    -(2\alpha_A-\varepsilon)\|w\|_V^2
    \\
    &\quad
    +
    \bigl(
        \varepsilon^{-1}C_{\mathrm{rc}}^2
        +2L_{F}
        +2\kappa_K
    \bigr)
    \|w\|_H^2.
\end{aligned}
\end{equation}
Further, using ~\eqref{eq:Z-lip},  we can write  
\begin{equation}
\label{eq:mon_esti3}
    \|\mathbf B(t,u)-\mathbf B(t,v)\|_{\mathcal L_2^0}^2
    \leq \|Z(t,u)-Z(t,v)\|_{\mathcal L_2^0}^2 \leq
    L_Z^2\|w\|_H^2 .
\end{equation}
Adding this estimate to \eqref{eq:mon_esti2} yields
\[
\begin{aligned}
    &2\langle \mathbf A(t,u)-\mathbf A(t,v),w\rangle_{V',V}
    +
    \|\mathbf B(t,u)-\mathbf B(t,v)\|_{\mathcal L_2^0}^2
    \leq \\ &
    \bigl(
        \varepsilon^{-1}C_{\mathrm{rc}}^2
        +2L_{F}
        +2\kappa_K
        +L_Z^2
    \bigr)
    \|w\|_H^2 \leq
    \eta\|w\|_H^2,
\end{aligned}
\]
since $\mathfrak w_1=2L_Z^2$. This is equivalent to
\eqref{eq:Liu-monotonicity}.

\smallskip
\noindent\textbf{\ref{ass:Liu-hemicontinuity}~Hemicontinuity.}
Fix $u,v,w\in V$. The map
\[
    \phi
    \mapsto
    \langle A(u+\phi v),w\rangle_{V',V}
    +
    \langle A_{\mathrm{rc}}(t)(u+\phi v),w\rangle_{V',V}
    +
    \langle U_N^\diamond K(u+\phi v),w\rangle_H
\]
is affine in $\phi$. Moreover, by the continuity of
$F(t,\cdot)\colon V\to H$, the mapping 
\[
    \phi
    \mapsto
    \langle F(t,u+\phi v),w\rangle_H
\]
is continuous. Hence  $\phi \mapsto  \langle \mathbf A(t,u+\phi v),w\rangle_{V',V}$ is continuous.

\smallskip
\noindent\textbf{\ref{ass:Liu-Lipschitz}~Lipschitz continuity of $\mathbf B$.} This follows by using \eqref{eq:mon_esti3} and the continuous embedding $V\hookrightarrow H$ (see
\eqref{eq:embed-contractive}). Thus condition~\ref{ass:Liu-Lipschitz} holds with $L_{\mathbf B}=L_Z$. The proof is complete.
\end{proof}

With Assumption~\ref{ass:Liu} verified, the abstract equation~\eqref{eq:abstract-spde} falls within the variational framework of \cite[Sec.~1.3.1]{Liu06}. We next recall the corresponding notion of strong variational solution and then state the resulting global well-posedness property.

\begin{definition}[Strong variational solution {\cite[Def.~1.3.1]{Liu06}}]
\label{def:strong-solution}
Let $T>0$ and $p\geq1$. An $H$-valued $(\mathcal F_t)_{t\geq0}$-adapted process $X=(X_t)_{t\in[0,T]}$ is called a \emph{strong variational solution} of~\eqref{eq:Liu-integral} on $[0,T]$ if $X\in L^p\bigl(\Omega;L^p(0,T;V)\bigr)$ and, for every $t\in[0,T]$, the identity 
\begin{equation}\label{eq:strog_sol_def}
    X_t = X_0 + \int_0^t \mathbf A(s,X_s)\,\mathrm ds + \int_0^t \mathbf B(s,X_s)\,\mathrm dW_s \quad \text{ in } V'
\end{equation}
holds $\mathbb P$-almost surely. If the same property holds on every finite interval $[0,T]$, then $X$ is called a \emph{global strong variational solution}.
\end{definition}

\begin{proposition}[Global well-posedness]
\label{prop:wp}
Suppose that Assumption~\ref{ass:standing} holds. Then, for every $X_0\in L^2(\Omega,\mathcal F_0,\mathbb P;H)$, there exists a unique global strong variational solution $X$ of~\eqref{eq:abstract-spde}. Moreover, for every $T>0$,
\begin{enumerate}
    \item $X$ is $(\mathcal F_t)_{t\geq0}$-progressively measurable on
    $[0,T]$;
    \item $X\in L^2(\Omega;L^2(0,T;V))$;
    \item $X(\cdot,\omega)\in C([0,T];H)$ for $\mathbb P$-a.e.\
    $\omega$.
\end{enumerate}
In particular, $X(\cdot,\omega)\in C([0,\infty);H)$ for $\mathbb P$-a.e. $\omega$.
\end{proposition}

\begin{proof}
Fix $T>0$. By Proposition~\ref{prop:Liu-verification}, the coefficients $\mathbf A$ and $\mathbf B$ satisfy the hypotheses of \cite[Thm.~1.3.1]{Liu06} on $[0,T]$. Hence there exists a unique strong variational solution $X^T$ of~\eqref{eq:Liu-integral} on $[0,T]$. In addition, $X^T$ is progressively measurable, belongs to $L^2(\Omega;L^2(0,T;V))$, and has $H$-continuous trajectories almost surely.

If $0<S<T$, then the restrictions of $X^S$ and $X^T$ to $[0,S]$ are both strong variational solutions on $[0,S]$. By uniqueness in \cite[Thm.~1.3.1]{Liu06}, they are indistinguishable on $[0,S]$. Hence the family $(X^T)_{T>0}$ is consistent on overlaps and defines a unique global process $X$ on $[0,\infty)$ with the stated properties on every finite interval.

Finally, for each $n\in\mathbb N$, let $\Omega_n := \{\omega\in\Omega:\ X(\cdot,\omega)\in C([0,n];H)\}$. Then $\mathbb P(\Omega_n)=1$ for every $n$, and therefore $\mathbb P\Bigl(\bigcap_{n=1}^\infty\Omega_n\Bigr)=1$. On this full-measure set, $X(\cdot,\omega)\in C([0,n];H)$ for every $n\in\mathbb N$, hence $X(\cdot,\omega)\in C([0,\infty);H)$. This proves the claim.
\end{proof}

\subsection{It\^o formula}
\label{subsec:ito-formula}

We use the following variational It\^o formula in the stability analysis. It applies to the strong variational solution $X$ of~\eqref{eq:Liu-integral} in the Gelfand triple~\eqref{eq:gelfand-triple}. In particular, we shall apply it to the exponentially weighted energy $\Psi(t,v)=e^{2\mu t}\|v\|_H^2$ for $\mu>0$ and, after a localization argument, to the logarithmic energy $\Psi(v)=\log\|v\|_H^2$ for $v\in H\setminus\{0\}$. These identities will be used in Section~\ref{sec:feed-stab} in the proof of the mean-square and pathwise stabilization results of Theorem~\ref{thm:main}. The logarithmic argument leading to pathwise stabilization is inspired by \cite[Thm.~3.9.1]{Liu06}.

For later use, we recall the standard spaces associated with the stochastic integral. For $T>0$, we denote by $\mathcal N_W^2(0,T;\mathcal L_2^0)$ the space of all predictable processes  $\Phi\colon\Omega\times[0,T]\to\mathcal L_2^0$ such that
\[
    \mathbb E\int_0^T \|\Phi(s)\|_{\mathcal L_2^0}^2\,\mathrm ds<\infty,
\]
and by $\mathcal N_W(0,T;\mathcal L_2^0)$ the space of all predictable processes $\Phi\colon\Omega\times[0,T]\to\mathcal L_2^0$ such that
\[
    \mathbb P\left(
        \int_0^T \|\Phi(s)\|_{\mathcal L_2^0}^2\,\mathrm ds<\infty
    \right)=1.
\]

\begin{lemma}[Variational It\^o formula {\cite[Thm.~1.3.3]{Liu06}}]
\label{lem:ito}
Let $T>0$, and let $X$ be the strong variational solution
of~\eqref{eq:Liu-integral} on $[0,T]$. Let
$\Psi\colon[0,T]\times H\to\mathbb R$ satisfy the following conditions:
\begin{enumerate}[label=\textup{(\roman*)}]

    \item\label{cond:ito-i}
    for every $v\in V$, the map $t\mapsto\Psi(t,v)$ is differentiable
    on $[0,T]$, and for every $t\in[0,T]$, the map
    $v\mapsto\Psi(t,v)$ is twice Fr\'echet differentiable on $H$;

    \item\label{cond:ito-ii}
    the functions $\partial_t\Psi$, $D\Psi$, and $D^2\Psi$ are locally
    bounded on $[0,T]\times H$, and $\Psi$ and $D\Psi$ are continuous
    on $[0,T]\times H$. Here, $D\Psi(t,v)$ and $D^2\Psi(t,v)$ denote the first and second Fr\'echet derivatives of $\Psi$ with respect to the state variable $v$.

    \item\label{cond:ito-iii}
    for every trace-class operator $S\colon H\to H$, the map
    \[
        (t,v)\mapsto
        \operatorname{Tr}\bigl(SD^2\Psi(t,v)\bigr)
    \]
    is continuous on $[0,T]\times H$;

    \item\label{cond:ito-iv}
    if $v\in V$, then $D\Psi(t,v)\in V$ for every $t\in[0,T]$, and
    for every $\varphi\in V'$, the map
    \[
        (t,v)\mapsto
        \langle \varphi,D\Psi(t,v)\rangle_{V',V}
    \]
    is continuous on $[0,T]\times V$;

    \item\label{cond:ito-v}
    there exists a constant $C_{\Psi,T}>0$ such that
    \[
        \|D\Psi(t,v)\|_V
        \leq
        C_{\Psi,T}\bigl(1+\|v\|_V\bigr)
        \qquad
        \forall\,t\in[0,T],\ \forall\,v\in V .
    \]

\end{enumerate}
Then, $\mathbb P$-almost surely, the identity
\begin{equation}\label{eq:ito-general}
\begin{aligned}
    \Psi(t,X_t)
    &=
    \Psi(0,X_0)
    +
    \int_0^t
    \bigl\langle
        D\Psi(s,X_s),
        \mathbf B(s,X_s)\,\mathrm dW_s
    \bigr\rangle_H
    \\
    &\quad
    +
    \int_0^t
    \Bigl(
        \partial_t\Psi(s,X_s)
        +
        \langle\mathbf A(s,X_s),D\Psi(s,X_s)\rangle_{V',V}
    \Bigr)\,\mathrm ds
    \\
    &\quad
    +
    \frac12
    \int_0^t
    \operatorname{Tr}\Bigl(
        D^2\Psi(s,X_s)
        \mathbf B(s,X_s)\mathbf B(s,X_s)^*
    \Bigr)\,\mathrm ds
\end{aligned}
\end{equation}
holds for all $t\in[0,T]$.
\end{lemma}

\begin{proof}
We fix $T>0$. Due to Proposition~\ref{prop:wp}, the solution $X$ is progressively measurable on $[0,T]$, belongs to $L^2(\Omega;L^2(0,T;V))$, has $H$-continuous trajectories almost surely, and satisfies \eqref{eq:strog_sol_def} for all $t\in[0,T]$, almost surely. Using~\eqref{eq:Liu-growth}, we can infer that 
\[
    \mathbf A(\cdot,X)\in L^2(\Omega;L^2(0,T;V')).
\]
Moreover, $\mathbf B(\cdot,X)$ is predictable and, by \eqref{eq:trace-bound} and the continuous embedding $V\hookrightarrow H$, we obtain that 
\[
    \mathbb E\int_0^T
    \|\mathbf B(s,X_s)\|_{\mathcal L_2^0}^2\,\mathrm ds
    <\infty.
\]
Thus, we have $\mathbf B(\cdot,X)\in \mathcal N_W^2(0,T;\mathcal L_2^0)$. Since $V$ and $V'$ are Hilbert spaces, the hypotheses of \cite[Thm.~1.3.3]{Liu06} and the subsequent remark are satisfied, and \eqref{eq:ito-general} follows.
\end{proof}

\begin{corollary}[Exponentially weighted energy identity]
\label{cor:ito-energy}
Under the hypotheses of Lemma~\ref{lem:ito}, for every $\mu>0$, $T>0$, and $\mathbb P$-almost surely, the following identity holds for all $t\in[0,T]$:
\begin{equation}\label{eq:ito-energy}
\begin{aligned}
    &e^{2\mu t}\|X_t\|_H^2
    =
    \|X_0\|_H^2
    \\
    &\quad
    +
    \int_0^t e^{2\mu s}
    \Bigl[
        2\mu\|X_s\|_H^2
        +
        2\langle\mathbf A(s,X_s),X_s\rangle_{V',V}
        +
        \|\mathbf B(s,X_s)\|_{\mathcal L_2^0}^2
    \Bigr]\,\mathrm ds
    +
    M(t),
\end{aligned}
\end{equation}
where
\begin{equation}\label{eq:martingale-def}
    M(t)
    :=
    2\int_0^t e^{2\mu s}
    \langle X_s,\mathbf B(s,X_s)\,\mathrm dW_s\rangle_H .
\end{equation}
\end{corollary}

\begin{proof} 
We apply Lemma~\ref{lem:ito} with
\[
    \Psi(t,v):=e^{2\mu t}\|v\|_H^2,
    \qquad (t,v)\in[0,T]\times H .
\]
Then we have
\[
    \partial_t\Psi(t,v)
    =
    2\mu e^{2\mu t}\|v\|_H^2 .
\]
Moreover, the first Fr\'echet derivative with respect to $v$ is the
functional $D\Psi(t,v)\in H'$ given by
\[
    D\Psi(t,v)w
    =
    2e^{2\mu t}\inp[H]{v}{w},
    \qquad w\in H .
\]
The second Fr\'echet derivative is the operator
$D^2\Psi(t,v)\in\mathcal L(H,H')$ defined by
\[
    \langle D^2\Psi(t,v)z,w\rangle_{H',H}
    =
    2e^{2\mu t}\inp[H]{z}{w},
    \qquad z,w\in H .
\]
The assumptions of Lemma~\ref{lem:ito} are immediate on every finite interval $[0,T]$. Using the fact that 
\[
    \operatorname{Tr}\Bigl(
        D^2\Psi(t,v)\mathbf B(t,v)\mathbf B(t,v)^*
    \Bigr)
    =
    2e^{2\mu t}\|\mathbf B(t,v)\|_{\mathcal L_2^0}^2,
\]
and substituting these expressions into~\eqref{eq:ito-general} gives us \eqref{eq:ito-energy} and~\eqref{eq:martingale-def}.
\end{proof}

\begin{corollary}[Logarithmic energy identity]
\label{cor:ito-log-energy}
Under the assumptions  of Lemma~\ref{lem:ito}, suppose that, in addition, the following holds
\begin{equation}\label{eq:nonvanishing-cor}
    \|X_t\|_H>0
    \qquad
    \text{for all }t\geq0,\ \mathbb P\text{-a.s.}
\end{equation}
Then, for every $T>0$ and $\mathbb P$-almost surely, we have
\begin{equation}\label{eq:ito-log-energy}
    \log\|X_t\|_H^2-\log\|X_0\|_H^2
    =
    \int_0^t \mathcal I_s\,\mathrm ds
    +
    \widetilde M_t,
    \qquad
    t\in[0,T],
\end{equation}
where
\begin{equation}\label{eq:log-terms}
\begin{aligned}
    \widetilde M_t
    &:=
    2\int_0^t
    \frac{
        \langle X_s,\mathbf B(s,X_s)\,\mathrm dW_s\rangle_H
    }{\|X_s\|_H^2},
    \\
    \mathcal I_s
    &:=
    \frac{
        2\langle\mathbf A(s,X_s),X_s\rangle_{V',V}
        +
        \|\mathbf B(s,X_s)\|_{\mathcal L_2^0}^2
    }{\|X_s\|_H^2}
    -
    \frac{
        2\|\mathbf B(s,X_s)^*X_s\|_{U_0}^2
    }{\|X_s\|_H^4}.
\end{aligned}
\end{equation}
\end{corollary}

\begin{proof}
Fix $T>0$. For $\delta>0$, define the stopping time
\[
    \tau_\delta
    :=
    \inf\{t\in[0,T]:\|X_t\|_H\leq\delta\},
\]
with the convention $\tau_\delta=T$ if the set is empty. Applying the variational It\^o formula to a $C^2$-function on $H$ which agrees with $v\mapsto\log\|v\|_H^2$ on $\{v\in H:\|v\|_H\geq\delta\}$, we obtain the desired identity on $[0,\tau_\delta]$. On the full-measure set on which $X(\cdot)\in C([0,T];H)$ and $\|X_t\|_H>0$ for all $t\in[0,T]$, continuity gives $\inf_{t\in[0,T]}\|X_t\|_H>0$. Hence, for all sufficiently small $\delta>0$, we have $\tau_\delta=T$. Letting $\delta\downarrow0$, the stopped identity gives \eqref{eq:ito-log-energy}--\eqref{eq:log-terms} on $[0,T]$.
\end{proof}

\section{Feedback stabilization via oblique projections}
\label{sec:feed-stab}

In this section we construct an explicit finite-dimensional feedback law based on oblique projections. The feedback acts through the family~$\{\1_{\clo_i}\}_{i=1}^{N} \subset H$  of indicator actuators introduced above, in~\eqref{eq:intro-control-term}, and is designed so that the closed-loop principal part satisfies a strict coercivity estimate. This estimate is the key ingredient in the subsequent stability analysis.

We first recall the oblique-projection notation used in the construction and formulate the structural stabilizability condition. We then define the feedback operator $K_L^{[\lambda]}$ and explain why, for sufficiently large $\lambda$, it satisfies the required coercivity condition. Finally, in Section~\ref{subsec:exponential-stability}, we prove the main stabilization result: mean-square exponential stability in the general additive and multiplicative noise setting, and pathwise exponential stability in the pure multiplicative noise case.

\subsection{Oblique projections and feedback construction}
\label{subsec:feedback-operator}

We begin with the oblique projections used in the feedback design. Let $H$ be a Hilbert space with scalar product $\langle \cdot, \cdot \rangle_H$. For $B\subset H$, its orthogonal complement is denoted by 
\[
    B^\perp \coloneqq \{h\in H:\ \langle h, s \rangle_H=0\ \text{for all }s\in B\}.
\]
Let $\mathcal{S}$ and $\mathcal{T}$ be closed subspaces of $H$. If $H=\mathcal{S}+\mathcal{T}$ and $\mathcal{S}\cap\mathcal{T}=\{0\}$, then we write $H=\mathcal{S}\oplus\mathcal{T}$. In this case, every $h\in H$ has a unique decomposition
\[
    h=h_{\mathcal{S}}+h_{\mathcal{T}},
    \qquad
    (h_{\mathcal{S}},h_{\mathcal{T}})\in\mathcal{S}\times\mathcal{T} .
\]
The oblique projection onto $\mathcal{S}$ along $\mathcal{T}$ is the operator $P_{\mathcal{S}}^{\mathcal{T}}\in\mathcal{L}(H,\mathcal{S})$ defined by
\[
    P_{\mathcal{S}}^{\mathcal{T}}h\coloneqq h_{\mathcal{S}}.
\]
The symbols $\mathcal{S}$ and $\mathcal{T}$ are used here only as generic placeholders for closed subspaces. In the feedback construction below, the relevant subspaces are the actuator space \[\mathcal{U}_L\coloneqq\operatorname{span}\{\1_{\clo_i}\}_{i=1}^{N} \subset H\] and an appropriate auxiliary space $\widetilde{\mathcal{U}}_L$ we shall precise later on. 

The following estimate for the reaction--convection operator will be used in the stability proof. It exploits the direct-sum structure $\mathfrak X=\mathcal L(H,V')+\mathcal L(V,H)$ from Assumption~\ref{ass:standing}, item~\ref{ass:Arc}.

\begin{lemma}[{\cite[Lem.~3.1]{Rod21-sicon}}]
\label{lem:Arc-kappa}
Assume that Assumption~\ref{ass:standing}, item~\ref{ass:Arc}, holds. Then,  for every $\kappa>0$, for a.e.\ $t\geq0$, and for all $v,g\in V$, we have 
\begin{equation}\label{eq:Arc-kappa}
    2\left|
        \langle A_{\mathrm{rc}}(t)v,g\rangle_{V',V}
    \right|
    \leq
    \kappa\bigl(\|g\|_V^2+\|v\|_V^2\bigr)
    +
    \kappa^{-1}C_{\mathrm{rc}}^2
    \bigl(\|v\|_H^2+\|g\|_H^2\bigr).
\end{equation}
\end{lemma}

We now construct an oblique projection-based feedback. The stabilization condition is governed by the number of actuators, which must be taken sufficiently large. Following~\cite{KunRodWal21}, we encode this actuator dimension by a monotone increasing map $\vartheta:\mathbb N_+\to\mathbb N_+$ and set
\[
    N\coloneqq \vartheta(L), \qquad L\in\mathbb N_+.
\]
Thus $L$ is the indexing parameter of the feedback design, whereas $N$ denotes the associated number of actuators. Hence, we consider the family of input feedback operators
\begin{equation}\notag
    \left\{ K_L^{[\lambda]} \in \mathcal{L}(H,\mathbb{R}^{N})
    \;\middle|\;
    (L,\lambda) \in \mathbb N_+ \times \overline{\mathbb{R}}_+ \right\}.
\end{equation}
Given $\mu>0$, we seek an index pair $(L,\lambda)$ for which the solution of
\begin{equation}\label{sys-X-feed}
    \begin{split}
    \mathrm dX + \bigl(AX+A_{\mathrm{rc}}(t)X+F(t,X)\bigr)\,\mathrm dt &= \bigl(Z(t,X)+\Xi(t)\bigr)\,\mathrm dW + U_L^\diamond K_L^{[\lambda]} X \,\mathrm dt, \\
    X\rest{t=0} &= X_0 .
    \end{split}
\end{equation}
is stable. For a fixed $L\geq1$, let $\mathcal{U}_L=\operatorname{span}\{\1_{\clo_i}\}_{i=1}^{N} \subset H$ be the actuator space, and let $\widetilde{\mathcal{U}}_L\subset V$ be an $N$-dimensional auxiliary subspace satisfying the direct-sum condition
\begin{equation}\label{eq:direct-sum}
    H=\mathcal{U}_L\oplus\widetilde{\mathcal{U}}_L^\perp .
\end{equation}
Equivalently, every $h\in H$ can be decomposed uniquely as
\[
    h=h_{\mathcal{U}_L}+h_{\widetilde{\mathcal{U}}_L^\perp},
    \qquad
    h_{\mathcal{U}_L}\in\mathcal{U}_L,\quad
    h_{\widetilde{\mathcal{U}}_L^\perp}\in\widetilde{\mathcal{U}}_L^\perp .
\]
Condition~\eqref{eq:direct-sum} is the structural requirement on $\widetilde{\mathcal{U}}_L$. A concrete spectral choice is discussed later in  Section~\ref{sec:simulations}. Under~\eqref{eq:direct-sum}, the feedback is defined by means of two oblique projections. Since $U_L^\diamond\colon\mathbb R^{N}\to\mathcal{U}_L$ is an isomorphism onto the actuator space, we define $K_L^{[\lambda]}\in\mathcal{L}(H,\mathbb R^{N})$ equivalently through 
\begin{equation}\label{eq:KL-def}
   U_L^\diamond K_L^{[\lambda]}
    :=
    -\lambda\,
    P_{\mathcal{U}_L}^{\widetilde{\mathcal{U}}_L^\perp}  \circ   P_{\widetilde{\mathcal{U}}_L}^{\mathcal{U}_L^\perp},
    \qquad
    \lambda>0 .
\end{equation}
Thus $U_L^\diamond K_L^{[\lambda]}\in\mathcal{L}(H,\mathcal{U}_L)\subset\mathcal{L}(H)$, and $K_L^{[\lambda]}$ is the corresponding coordinate representation in $\mathbb R^{N}$. The offline computation of $K_L^{[\lambda]}$ reduces to the inversion of an $ N\times N$ matrix; see Section~\ref{subsec:galerkin}. 

We next state the coercivity condition required of the feedback operator. It is the abstract stabilizability property used in the proof of the main stability theorem.

\begin{assumption}\label{ass:stab-KM}
For every $\varrho>0$, there exists $L_*\in\mathbb N_+$ such that, for every $L\geq L_*$, there exists $\lambda_*(L)\in [0,\infty)$  with the following property: for all $\lambda\geq\lambda_*(L)$,
\begin{equation}\label{eq:stab-KM}
    2\langle (A-U_L^\diamond K_L^{[\lambda]})w,w\rangle_{V',V}
    \geq
    \|w\|_V^2+\varrho\|w\|_H^2
    \qquad
    \forall\,w\in V .
\end{equation}
\end{assumption}

The satisfiability of Assumption~\ref{ass:stab-KM}
shall be shown in Section~\ref{ssec:satsf-ass:stab-KM}.


\subsection{Exponential stability: Mean-Square and Almost-Sure}
\label{subsec:exponential-stability}

We now prove the main stabilization result. It provides the explicit closed-loop form of the stability statements~\eqref{eq:intro-goal-feed-Exp} and~\eqref{eq:intro-goal-feed-P} for system~\eqref{eq:sys-X}. The first part gives exponential stability in mean square in the presence of additive and multiplicative noise. The second part gives pathwise exponential stability in the pure multiplicative noise case.

\begin{theorem}[Stabilization]\label{thm:main}
    Let Assumptions \ref{ass:standing} and  \ref{ass:stab-KM} hold. Given any desired decay rate $\mu > 0$, there exists $L_*\in\mathbb N_+$ such that for every $L\geq L_*$, there exists $\lambda_*(L) \in [0,\infty)$ such that for every  $\lambda\ge \lambda_*(L)$, the solution $X$ of~\eqref{sys-X-feed} given by Proposition~\ref{prop:wp} satisfies the following stability bounds:
    
    \begin{enumerate}[label=(\roman*)]
        \item \textbf{Mean-Square Stabilization (General Noise):} 
        For all $t \geq 0$ and all $X_0 \in L^2(\Omega,\mathcal F_0,\mathbb P;H)$, we have 
        \begin{equation}\label{eq:main-exp}
            \mathbb{E} \| X_t \|_H^2 \leq e^{-2\mu t} \mathbb{E} \| X_0 \|_H^2 + \frac{\mathfrak{w}_0}{2\mu}.
        \end{equation}
        where $\mathfrak{w}_0$ is related to the additive noise $\Xi$.
    
        \item \textbf{Almost Sure Exponential Decay (Pure Multiplicative Noise):}
        If the noise is purely multiplicative (i.e., $\Xi \equiv0$, and hence $\mathfrak{w}_0 = 0$), and in addition the solution $X$ of \eqref{eq:abstract-spde} satisfies $\|X_t\|_H \neq 0$ for all $t\geq0$ $\mathbb{P}$-a.s., then the zero equilibrium is almost surely exponentially stable: every solution has a sample Lyapunov exponent at most $-2\mu$, i.e.
        \begin{equation}\label{eq:main-path-multiplicative}
            \limsup_{t\to\infty} \frac{1}{t} \log \| X_t \|_H^2 \leq -2\mu, \quad \mathbb{P}\text{-a.s.}
        \end{equation}
    \end{enumerate}
\end{theorem}

Before proving Theorem~\ref{thm:main}, we establish the following auxiliary lemma.

\begin{lemma}[$M(t)$ is a true martingale]\label{lem:martingale}
    Assume that Assumption~\ref{ass:standing} holds, and let $X$ be the global strong variational solution of~\eqref{eq:abstract-spde} given by Proposition~\ref{prop:wp}. Then, for every $\mu>0$ and $T>0$, the process $(M(t))_{t\in[0,T]}$ defined by~\eqref{eq:martingale-def} is an $(\mathcal F_t)_{t\in[0,T]}$-martingale. In particular,
    \[
        \mathbb E[M(t)]=0, \qquad t\in[0,T].
    \]
\end{lemma}
\begin{proof}
    Fix $T>0$ and $\mu>0$, and let $X$ be the global strong variational solution of~\eqref{eq:abstract-spde} given by Proposition~\ref{prop:wp}. Defining 
    \[
        \Phi(s)u
        :=
        2e^{2\mu s}\langle X_s,\mathbf B(s,X_s)u\rangle_H,
        \qquad
        u\in U_0,\ s\in[0,T].
    \]
    and using~\eqref{eq:martingale-def}, we can write 
    \[
        M_t=\int_0^t \Phi(s)\,\mathrm dW_s,
        \qquad t\in[0,T].
    \]
    By Proposition~\ref{prop:wp}, the process $X$ is progressively measurable and has $H$-continuous trajectories. Hence $X$ is predictable. By~\eqref{eq:AB-def} and Assumption~\ref{ass:standing}, items~\ref{ass:Z} and~\ref{ass:Xi}, the map $\mathbf B\colon [0,\infty)\times H\to \mathcal L_2^0$ is measurable. Since $X$ is predictable, it follows that $\Phi$ is a predictable $\mathcal L_2(U_0,\mathbb R)$-valued process. Moreover, for every $s\in[0,T]$, we can write
    \begin{equation}\label{eq:Phi-CS-rewrite}
        \|\Phi(s)\|_{\mathcal L_2(U_0,\mathbb R)}^2
        =
        4e^{4\mu s}\sum_{k=1}^\infty \lambda_k
        \langle X_s,\mathbf B(s,X_s)e_k\rangle_H^2
        \leq
        4e^{4\mu s}\|X_s\|_H^2\|\mathbf B(s,X_s)\|_{\mathcal L_2^0}^2 .
    \end{equation}
     We denote  $S:=\sup_{r\in[0,T]}\|X_r\|_H$. Then, since $X\in C([0,T];H)$ $\mathbb P$-a.s. by Proposition~\ref{prop:wp}, we have $S<\infty$ $\mathbb P$-a.s. Using~\eqref{eq:Phi-CS-rewrite} and the trace bound~\eqref{eq:trace-bound}, we obtain
    \[
        \int_0^T \|\Phi(s)\|_{\mathcal L_2(U_0,\mathbb R)}^2\,\mathrm ds
        \leq
        4e^{4\mu T}T\,S^2\bigl(\mathfrak w_1 S^2+\mathfrak w_0\bigr)
        <\infty
        \qquad \mathbb P\text{-a.s.}
    \]
    Thus,  $\Phi\in \mathcal N_W\bigl(0,T;\mathcal L_2(U_0,\mathbb R)\bigr)$, and \cite[Thm.~4.27]{da2014stochastic} yields that $(M_t)_{t\in[0,T]}$ is a continuous local $(\mathcal F_t)$-martingale.
    
    For $n\in\mathbb N$, we define
    \[
        \tau_n:=\inf\{t\in[0,T]:\|X_t\|_H\geq n\}\wedge T,
    \]
    where $r\wedge s\coloneqq\min\{r,s\}$ for two given real numbers~$r$ and~$s$. 
    Since $X$ is adapted and has $H$-continuous trajectories by Proposition~\ref{prop:wp}, each $\tau_n$ is an $(\mathcal F_t)$-stopping time. Moreover, $\tau_n\uparrow T$ $\mathbb P$-a.s., because $S<\infty$ $\mathbb P$-a.s. Let $\Phi^n:=\1_{[0,\tau_n]}\Phi$. On $\{s\leq\tau_n\}$ we have $\|X_s\|_H\leq n$, and therefore, by~\eqref{eq:Phi-CS-rewrite} and~\eqref{eq:trace-bound},
    \begin{equation*}
        \begin{split}
            \mathbb E\int_0^T \|\Phi^n(s)\|_{\mathcal L_2(U_0,\mathbb R)}^2\,\mathrm ds
            &\leq
            4e^{4\mu T}\,
            \mathbb E\int_0^T \|X_s\|_H^2\|\mathbf B(s,X_s)\|_{\mathcal L_2^0}^2
            \1_{\{s\leq\tau_n\}}\,\mathrm ds \\
            &\leq
            4e^{4\mu T}T\,n^2(\mathfrak w_1 n^2+\mathfrak w_0)
            <\infty.
        \end{split}
    \end{equation*}
    Hence $\Phi^n\in \mathcal N_W^2\bigl(0,T;\mathcal L_2(U_0,\mathbb R)\bigr)$. Due to \cite[Thm.~4.27]{da2014stochastic}, the process $\int_0^t \Phi^n(s)\,\mathrm dW_s$ is a true $(\mathcal F_t)$-martingale, and by \cite[Lem.~4.24 and Rem.~4.25]{da2014stochastic}, we can infer that 
    \[
        M_{t\wedge\tau_n}
        =
        \int_0^t \Phi^n(s)\,\mathrm dW_s,
        \qquad t\in[0,T].
    \]
 Thus, $(M_{t\wedge\tau_n})_{t\in[0,T]}$ is a true $(\mathcal F_t)$-martingale for every $n\in\mathbb N$.
    
     We now fix $0\leq s\leq t\leq T$. Since $M$ has continuous trajectories and $\tau_n\uparrow T$ $\mathbb P$-a.s., we have
    \[
        M_{t\wedge\tau_n}\to M_t
        \qquad\text{and}\qquad
        M_{s\wedge\tau_n}\to M_s
        \qquad \mathbb P\text{-a.s.}
    \]
    By~\eqref{eq:ito-energy}, evaluated at $t\wedge\tau_n$, we can write 
    \begin{equation*}
        \begin{split}
            M_{t\wedge\tau_n} &= e^{2\mu(t\wedge\tau_n)}\|X_{t\wedge\tau_n}\|_H^2 - \|X_0\|_H^2 \\
            &\phantom{=} - \int_0^{t\wedge\tau_n}e^{2\mu s} \Bigl(2\mu\|X_s\|_H^2 + 2\langle\mathbf{A}(s,X_s),X_s\rangle_{V',V} + \|\mathbf{B}(s,X_s)\|^2_{\mathcal{L}_2^0} \Bigr)\,\mathrm{d}s.
        \end{split}   
    \end{equation*}
    Therefore, we have $\mathbb P$-a.s. that $|M_{t\wedge\tau_n}| \leq Y$, where
    \[
        Y:=
        e^{2\mu T}\bigl(S^2+\|X_0\|_H^2\bigr)
        +
        e^{2\mu T}\int_0^T
        \Bigl(
            2\mu S^2
            +2|\langle \mathbf A(r,X_r),X_r\rangle_{V',V}|
            +\|\mathbf B(r,X_r)\|_{\mathcal L_2^0}^2
        \Bigr)\,\mathrm dr .
    \]
    We claim that $Y\in L^1(\Omega)$. Indeed, $S^2\in L^1(\Omega)$ by Proposition~\ref{prop:wp}, since $X\in L^2(\Omega;C([0,T];H))$, and $\|X_0\|_H^2\in L^1(\Omega)$ by the assumption on the initial datum in Proposition~\ref{prop:wp}. Further, by~\eqref{eq:Liu-growth}, we have
    \[
        |\langle \mathbf A(r,X_r),X_r\rangle_{V',V}|
        \leq
        \|\mathbf A(r,X_r)\|_{V'}\|X_r\|_V
        \leq
        c(1+\|X_r\|_V)\|X_r\|_V
        \leq
        c(1+\|X_r\|_V^2),
    \]
    and Proposition~\ref{prop:wp} gives $X\in L^2(\Omega;L^2(0,T;V))$. Finally, by~\eqref{eq:trace-bound} and~\eqref{eq:embed-contractive}, we can write
    \[
        \|\mathbf B(r,X_r)\|_{\mathcal L_2^0}^2
        \leq
        \mathfrak w_1\|X_r\|_H^2+\mathfrak w_0
        \leq
        \mathfrak w_1\|X_r\|_V^2+\mathfrak w_0,
    \]
    so the diffusion term is also integrable. Hence $Y\in L^1(\Omega)$. By dominated convergence, we obtain
    \[
        M_{t\wedge\tau_n}\to M_t
        \qquad\text{and}\qquad
        M_{s\wedge\tau_n}\to M_s
        \qquad\text{in }L^1(\Omega).
    \]
    Using the $L^1$-contractivity of conditional expectation and the martingale property of $M_{\cdot\wedge\tau_n}$, we obtain
    \[
        \mathbb E[M_t\mid\mathcal F_s]
        =
        \lim_{n\to\infty}\mathbb E[M_{t\wedge\tau_n}\mid\mathcal F_s]
        =
        \lim_{n\to\infty} M_{s\wedge\tau_n}
        =
        M_s
        \qquad \text{in }L^1(\Omega).
    \]
    Thus, $\mathbb E[M_t\mid\mathcal F_s]=M_s$ $\mathbb P$-a.s. for all $0\leq s\leq t\leq T$, and $M$ is an $(\mathcal F_t)$-martingale on $[0,T]$. Since $M_0=0$, we also have $\mathbb E[M_t]=0$ for every $t\in[0,T]$.
\end{proof}


We now prove Theorem~\ref{thm:main} using Corollaries~\ref{cor:ito-energy} and~\ref{cor:ito-log-energy} and Lemma~\ref{lem:martingale}.

\begin{proof}[Proof of Theorem~\ref{thm:main}]
Fix $\mu>0$, and choose
\begin{equation}\label{eq:varrho-choice-main-rewrite}
    \varrho:=6C_{\mathrm{rc}}^2+3L_{F}^2+2\mu+\mathfrak w_1 .
\end{equation}
Throughout the proof, the feedback term is chosen as $U_L^\diamond K_L^{[\lambda]}$, with $K_L^{[\lambda]}$ defined by \eqref{eq:KL-def}. By Assumption~\ref{ass:stab-KM}, there exists $L_*\in\mathbb N_+$ such that, for every $L\geq L_*$, there exists $\lambda_*(L)\in[0,\infty)$ with the property that
\begin{equation}\label{eq:stab-KM-main-rewrite}
    2\langle (A-U_L^\diamond K_L^{[\lambda]})w,w\rangle_{V',V}
    \geq
    \|w\|_V^2+\varrho\|w\|_H^2
    \qquad \forall\,w\in V,
\end{equation}
for all $\lambda\geq \lambda_*(L)$. Fix such $L$ and $\lambda$, and let $X$ be the corresponding solution of~\eqref{sys-X-feed}.

We first prove~\eqref{eq:main-exp}. Fix $T>0$. By Corollary~\ref{cor:ito-energy}, for $\mathbb P$-a.e.\ sample path and all $t\in[0,T]$,
\begin{equation}\label{eq:ito-energy-main-rewrite}
\begin{aligned}
    e^{2\mu t}\|X_t\|_H^2
    &=
    \|X_0\|_H^2
    +M_t
    \\
    &\quad
    +
    \int_0^t e^{2\mu s}
    \Bigl(
        2\mu\|X_s\|_H^2
        +2\langle \mathbf A(s,X_s),X_s\rangle_{V',V}
        +\|\mathbf B(s,X_s)\|_{\mathcal L_2^0}^2
    \Bigr)\,\mathrm ds,
\end{aligned}
\end{equation}
where $M$ is the martingale from~\eqref{eq:martingale-def}. Using~\eqref{eq:AB-def}, we write
\begin{equation*}
    \begin{split}
        2\langle \mathbf A(s,X_s),X_s\rangle_{V',V}
        &=
        -2\langle (A-U_L^\diamond K_L^{[\lambda]})X_s,X_s\rangle_{V',V}
        -2\langle A_{\mathrm{rc}}(s)X_s,X_s\rangle_{V',V} \\&\phantom{=\;}
        -2\langle F(s,X_s),X_s\rangle_H.
    \end{split}
\end{equation*}
From~\eqref{eq:stab-KM-main-rewrite}, it follows that
\[
    -2\langle (A-U_L^\diamond K_L^{[\lambda]})X_s,X_s\rangle_{V',V}
    \leq
    -\|X_s\|_V^2-\varrho\|X_s\|_H^2 .
\]
Moreover, Lemma~\ref{lem:Arc-kappa} with $\kappa=\tfrac13$ gives
\[
    -2\langle A_{\mathrm{rc}}(s)X_s,X_s\rangle_{V',V}
    \leq
    \bigl|2\langle A_{\mathrm{rc}}(s)X_s,X_s\rangle_{V',V}\bigr|
    \leq
    \tfrac23\|X_s\|_V^2+6C_{\mathrm{rc}}^2\|X_s\|_H^2 .
\]
Further, since $F(s,0)=0$ and~\eqref{eq:N-lip} holds, we have
\[
    \|F(s,X_s)\|_H
    \leq
    L_{F}\|X_s\|_H
    \leq
    L_{F}\|X_s\|_V
\]
by~\eqref{eq:embed-contractive}. Hence Young's inequality with parameter $\tfrac13$ yields
\[
    -2\langle F(s,X_s),X_s\rangle_H
    \leq
    2\|F(s,X_s)\|_H\|X_s\|_H
    \leq
    \tfrac13\|X_s\|_V^2+3L_{F}^2\|X_s\|_H^2 .
\]
Combining the last three estimates and using~\eqref{eq:varrho-choice-main-rewrite}, we obtain
\begin{equation}\label{eq:drift-bound-main-rewrite}
    2\langle \mathbf A(s,X_s),X_s\rangle_{V',V}
    \leq
    -(2\mu+\mathfrak w_1)\|X_s\|_H^2 .
\end{equation}
Therefore, by~\eqref{eq:trace-bound}, we have
\begin{equation}\label{eq:integrand-bound-main-rewrite}
\begin{aligned}
    &2\mu\|X_s\|_H^2
    +2\langle \mathbf A(s,X_s),X_s\rangle_{V',V}
    +\|\mathbf B(s,X_s)\|_{\mathcal L_2^0}^2
    \\
    &\leq
    2\mu\|X_s\|_H^2
    -(2\mu+\mathfrak w_1)\|X_s\|_H^2
    +\mathfrak w_1\|X_s\|_H^2+\mathfrak w_0
    =
    \mathfrak w_0 .
\end{aligned}
\end{equation}
Inserting~\eqref{eq:integrand-bound-main-rewrite} into~\eqref{eq:ito-energy-main-rewrite}, taking expectations, and using Lemma~\ref{lem:martingale}, we get for every $t\in[0,T]$,
\[
    e^{2\mu t}\mathbb E\|X_t\|_H^2
    \leq
    \mathbb E\|X_0\|_H^2
    +\mathfrak w_0\int_0^t e^{2\mu s}\,\mathrm ds
    =
    \mathbb E\|X_0\|_H^2
    +\frac{\mathfrak w_0}{2\mu}(e^{2\mu t}-1).
\]
Thus, we have
\[
    \mathbb E\|X_t\|_H^2
    \leq
    e^{-2\mu t}\mathbb E\|X_0\|_H^2
    +\frac{\mathfrak w_0}{2\mu}(1-e^{-2\mu t})
    \leq
    e^{-2\mu t}\mathbb E\|X_0\|_H^2+\frac{\mathfrak w_0}{2\mu}
\]
for all $t\in[0,T]$. Since $T>0$ was arbitrary, this proves~\eqref{eq:main-exp} for all $t\geq0$.

We now prove~\eqref{eq:main-path-multiplicative}. Assume that $\Xi\equiv0$, so that $\mathfrak w_0=0$, and that $\|X_t\|_H>0$ for all $t\geq0$ almost surely. Fix again $T>0$. By Corollary~\ref{cor:ito-log-energy}, for $\mathbb P$-a.e.\ sample path and all $t\in[0,T]$, we can write 
\begin{equation}\label{eq:log-identity-main-rewrite}
    \log\|X_t\|_H^2-\log\|X_0\|_H^2
    =
    \int_0^t \mathcal I_s\,\mathrm ds+\widetilde M_t,
\end{equation}
with $\mathcal I_s$ and $\widetilde M_t$ defined by~\eqref{eq:log-terms}. Since the correction term in~\eqref{eq:log-terms} is nonpositive, \eqref{eq:drift-bound-main-rewrite} and $\mathfrak w_0=0$ imply
\[
    \mathcal I_s
    \leq
    \frac{
        2\langle \mathbf A(s,X_s),X_s\rangle_{V',V}
        +
        \|\mathbf B(s,X_s)\|_{\mathcal L_2^0}^2
    }{\|X_s\|_H^2}
    \leq
    -2\mu .
\]
Hence, we have 
\begin{equation}\label{eq:log-pathwise-main-rewrite}
    \log\|X_t\|_H^2
    \leq
    \log\|X_0\|_H^2+\widetilde M_t-2\mu t,
    \qquad t\in[0,T].
\end{equation}
We next define
\[
    \widetilde\Phi_s(u)
    :=
    \frac{2\langle X_s,\mathbf B(s,X_s)u\rangle_H}{\|X_s\|_H^2},
    \qquad u\in U_0.
\]
By Proposition~\ref{prop:wp}, the process $X$ is progressively measurable and has $H$-continuous trajectories and, thus, it is predictable. By~\eqref{eq:AB-def} and ~\ref{ass:Z} and~\ref{ass:Xi} from Assumption~\ref{ass:standing}, the map $\mathbf B\colon[0,\infty)\times H\to\mathcal L_2^0$ is measurable. Therefore $\widetilde\Phi$ is a predictable $\mathcal L_2(U_0,\mathbb R)$-valued process. Moreover, by Cauchy--Schwarz in $H$ and~\eqref{eq:trace-bound}, we can write
\[
    \|\widetilde\Phi_s\|_{\mathcal L_2(U_0,\mathbb R)}^2
    =
    \frac{4\|\mathbf B(s,X_s)^*X_s\|_{U_0}^2}{\|X_s\|_H^4}
    \leq
    \frac{4\|\mathbf B(s,X_s)\|_{\mathcal L_2^0}^2\|X_s\|_H^2}{\|X_s\|_H^4}
    \leq
    4\mathfrak w_1,
\]
where we have used the fact that $\mathfrak w_0=0$. Hence, we have
\[
    \mathbb E\int_0^T
    \|\widetilde\Phi_s\|_{\mathcal L_2(U_0,\mathbb R)}^2\,\mathrm ds
    \leq
    4\mathfrak w_1T<\infty,
\]
and, as a consequence,  $\widetilde\Phi\in \mathcal N_W^2(0,T;\mathcal L_2(U_0,\mathbb R))$. By~\cite[Thm.~4.27]{da2014stochastic}, $\widetilde M_t=\int_0^t\widetilde\Phi_s\,\mathrm dW_s$ is a continuous square-integrable martingale on $[0,T]$, and
\[
    [\widetilde M]_t
    =
    \int_0^t
    \|\widetilde\Phi_s\|_{\mathcal L_2(U_0,\mathbb R)}^2\,\mathrm ds
    \leq
    4\mathfrak w_1 t,
    \qquad t\in[0,T].
\]
Since $T>0$ is arbitrary, the same formula defines a continuous local martingale on $[0,\infty)$, still denoted by $\widetilde M$, and the bound
\begin{equation}\label{eq:QV-main-rewrite}
    [\widetilde M]_t\leq 4\mathfrak w_1 t
    \qquad\text{for all }t\geq0,\ \mathbb P\text{-a.s.}
\end{equation}
holds. Finally, dividing ~\eqref{eq:log-pathwise-main-rewrite} by $t>0$, letting $t\to\infty$, and using ~\eqref{eq:QV-main-rewrite}, we can infer that 
\[
    \limsup_{t\to\infty}\frac{[\widetilde M]_t}{t}
    \leq
    4\mathfrak w_1
    \qquad \mathbb P\text{-a.s.},
\]
and  the strong law of large numbers for continuous local martingales \cite[Thm.~1.3.4]{mao2007} gives
\[
    \lim_{t\to\infty}\frac{\widetilde M_t}{t}=0
    \qquad \mathbb P\text{-a.s.}
\]
Therefore, we have
\[
    \limsup_{t\to\infty}\frac1t\log\|X_t\|_H^2
    \leq
    \limsup_{t\to\infty}\frac{\log\|X_0\|_H^2}{t}
    +
    \limsup_{t\to\infty}\frac{\widetilde M_t}{t}
    -2\mu
    =
    -2\mu
\]
almost surely, which proves~\eqref{eq:main-path-multiplicative}.
\end{proof}

\subsection{Satisfiability of Assumption~\ref{ass:stab-KM} in Polygonal Domains}\label{ssec:satsf-ass:stab-KM} It remains to show that Assumption~\ref{ass:stab-KM} is satisfiable. For this purpose, we define the actuators as indicator functions constructed as follows. For a rectangular/box domain we can construct the actuators as in~\cite[Sect.~5]{KunRodWal21}
as illustrated in Fig.~\ref{fig.suppAct}.
\begin{figure}[htbp]%
    \centering%
    \subfigure[The supports of the actuators  in a rectangular domain. $N=L^2$\label{fig.suppActR}]
    {\includegraphics[width=1\textwidth]{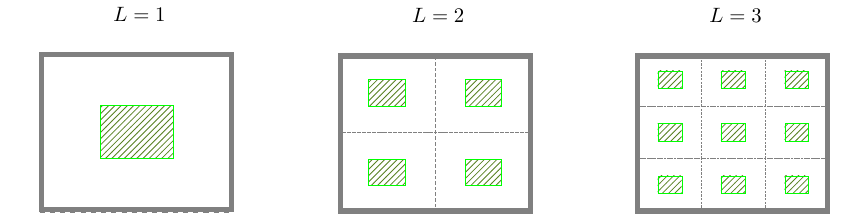}}
    \subfigure[The supports of the actuators in a triangular domain. $N=4^{L-1}$\label{fig.suppActT}]
    {\includegraphics[width=1\textwidth]{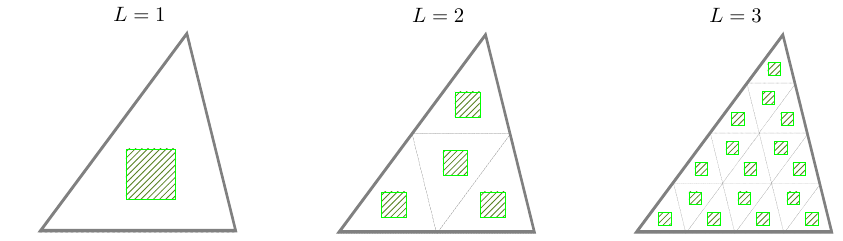}}
    \caption{Supports of the actuators in planar domains~$\mathcal{O}\subset\mathbb{R}^d$, $d=2$.} \label{fig.suppAct}
\end{figure}
Then, we take a family of feedback input operators of the from \eqref{eq:KL-def}, where~$\widetilde{\mathcal{U}}_L$ is either an auxiliary subspace of~$V$ satisfying the direct sum~$H=\mathcal{U}_L\oplus\widetilde{\mathcal{U}}_L^\perp$ or simply~$\widetilde{\mathcal{U}}_L=\mathcal{U}_L$. By ~\cite[Lem.~3.5 and Cor.~3.1]{KunRodWal21} we have that for any given~$\varrho>0$ we can find~$L=L(\varrho)$ and~$\lambda=\lambda(\varrho,L)$ so that
\begin{align}
    \|w\|_{V}^2 + \lambda \|P_{\widetilde{\mathcal{U}}_L}^{\mathcal{U}_L^\perp} w\|_H^2\ge \varrho\|w\|_{H}^2. \notag
\end{align}
Hence~$2\langle (A -U_L^\diamond K_L^{[\lambda]})w,w\rangle_{V',V}\ge \|w\|_{V}^2+\lambda \|P_{\widetilde{\mathcal{U}}_L}^{\mathcal{U}_L^\perp}w\|_{H}^2+\varrho\|w\|_{H}^2$, which implies that Assumption~\ref{ass:stab-KM} holds true for a rectangular/box domain. See also~\cite[Lem.~2.6]{AzmiKunRod23}.

The key point leading to the the results in~\cite[Lem.~3.5 and Cor.~3.1]{KunRodWal21} is that a rectangular domain with an arbitrary actuator region in its interior as in the reference initial  configuration~$L=1$, as in Fig.~\ref{fig.suppActR} above, can be partitioned into rescaled copies of itself. In particular, the total volume is kept constant, independent of~$L$. Furthermore, as noted in~\cite[Rem.~2.8]{AzmiKunRod23}, (using simplex partition results from~\cite[Sect.~3]{EdelsbrunnerGrayson00},  \cite[Thm.~4.1]{Bey00}), the argument can be extended for simplexes. Consequently,  Assumption~\ref{ass:stab-KM} holds true for a general polygonal domains.


\subsection{Application to observer design and data assimilation}
\label{subsec:observer-interpretation}

We briefly discuss the distinct and equally important problem of state estimation in applications, including observer design and continuous data assimilation~\cite{AzouaniOlsonTiti14}. We explain how the oblique-projection construction can be applied in the presence of additive noise in the state equation and stochastic measurement errors~\cite{BessaihOlsonTiti15}. We consider the dynamics
\begin{equation}\label{eq:obs-nominal}
\begin{aligned}
    \mathrm dX
    + \bigl(AX+A_{\mathrm{rc}}(t)X+F(t,X)\bigr)\,\mathrm dt
    &=
    \Xi(t)\,\mathrm dW ,
\end{aligned}
\end{equation}
where the initial state $X(0)$ is not known. Let $\mathcal O_1,\dots,\mathcal O_N\subset\mathcal O$ denote the supports of the sensors and set 
\[
    \psi_i:=\1_{\mathcal O_i},
    \qquad i=1,\dots,N .
\]
The sensors measure localized integrals of the unknown state,
\[
    \inp[H]{X(t)}{\psi_i}
    =
    \int_{\mathcal O_i}X(t,x)\,\mathrm dx,
    \qquad i=1,\dots,N .
\]
This leads to the output operator $\clC_N\colon H\to\mathbb R^N$, defined by
\[
    (\clC_Nv)_i:=\inp[H]{v}{\psi_i},
    \qquad i=1,\dots,N .
\]
The corresponding noiseless output is
\begin{equation}\label{eq:obs-output}
    Y_N^0(t)\coloneqq\clC_NX(t).
\end{equation}
We model measurement errors by the noisy observation process
\[
    \mathrm dY_N(t)
    =
    \clC_NX(t)\,\mathrm dt
    +
    R_N(t)\,\mathrm dW^{\rm obs}(t),
\]
where $W^{\rm obs}$ is an $\mathbb R^N$-valued standard Brownian motion, independent of $W$, and
\[
    R_N\in L^\infty([0,\infty);\mathcal L(\mathbb R^N)).
\]
We consider the Luenberger observer associated with \eqref{eq:obs-nominal}--\eqref{eq:obs-output},
\begin{equation}\label{eq:observer-formal}
\begin{aligned}
    \mathrm d\widehat X
    +\bigl(
        A\widehat X
        +A_{\rm rc}(t)\widehat X
        +F(t,\widehat X)
    \bigr)\,\mathrm dt
    &=
    \clL_N\bigl(
        \clC_N\widehat X(t)\,\mathrm dt-\mathrm dY_N(t)
    \bigr),
\end{aligned}
\end{equation}
where $\clL_N\colon\mathbb R^N\to H$ is an observer injection operator. Substituting the observation equation into \eqref{eq:observer-formal} yields
\begin{equation}\label{eq:observer-substituted}
\begin{aligned}
    \mathrm d\widehat X
    +\bigl(
        A\widehat X
        +A_{\rm rc}(t)\widehat X
        +F(t,\widehat X)
        -\clL_N\clC_N(\widehat X-X)
    \bigr)\,\mathrm dt
    &=
    -\clL_NR_N(t)\,\mathrm dW^{\rm obs}(t).
\end{aligned}
\end{equation}
We first observe that the coupled state--observer system is well posed within the same variational framework. To this end, we denote
\[
    \boldsymbol V:=V\times V,
    \qquad
    \boldsymbol H:=H\times H,
    \qquad
    \boldsymbol V^*:=V^*\times V^*.
\]
For $\boldsymbol X:=(X,\widehat X)^T$, we also define
\[
    \boldsymbol A\boldsymbol X
    :=
    (AX,A\widehat X)^T,
    \qquad
    \boldsymbol A_{\rm rc}(t)\boldsymbol X
    :=
    \bigl(A_{\rm rc}(t)X,A_{\rm rc}(t)\widehat X\bigr)^T,
\]
and
\[
    \boldsymbol G(t,\boldsymbol X)
    :=
    \begin{pmatrix}
        F(t,X)\\[1mm]
        F(t,\widehat X)-\clL_N\clC_N(\widehat X-X)
    \end{pmatrix}.
\]
Moreover, let
\[
    \widetilde U:=U\times\mathbb R^N,
    \qquad
    \widetilde Q:=\operatorname{diag}(Q,I_{\mathbb R^N}),
    \qquad
    \widetilde W:=(W,W^{\rm obs}),
\]
then $\widetilde W$ is a $\widetilde Q$-Wiener process on $\widetilde U$, and its associated covariance space is given by
\[
    \widetilde U_0
    :=
    \widetilde Q^{1/2}\widetilde U
    =
    U_0\times\mathbb R^N.
\]
Then, by defining 
\[
    \boldsymbol\Xi(t)(u,\eta)
    :=
    \begin{pmatrix}
        \Xi(t)u\\[1mm]
        -\clL_NR_N(t)\eta
    \end{pmatrix},
    \qquad
    (u,\eta)\in\widetilde U_0,
\]
the coupled system \eqref{eq:obs-nominal} and \eqref{eq:observer-substituted} can then be written as
\begin{equation}\label{eq:coupled-state-observer}
    \mathrm d\boldsymbol X
    +
    \bigl(
        \boldsymbol A\boldsymbol X
        +\boldsymbol A_{\rm rc}(t)\boldsymbol X
        +\boldsymbol G(t,\boldsymbol X)
    \bigr)\,\mathrm dt
    =
    \boldsymbol\Xi(t)\,\mathrm d\widetilde W.
\end{equation}
Since $F(t,\cdot)\colon V\to H$ is globally Lipschitz with respect to the $H$-norm, the map $\boldsymbol G(t,\cdot)\colon\boldsymbol V\to\boldsymbol H$ is globally Lipschitz with respect to the $\boldsymbol H$-norm. Moreover, $\boldsymbol G(t,0)=0$. Further, if the following holds, 
\begin{equation}\label{eq:bounded_coupled_noise}
    \Xi\in L^\infty([0,\infty);\mathcal L_2(U_0,H)),
    \qquad
    \clL_NR_N
    \in
    L^\infty([0,\infty);\mathcal L_2(\mathbb R^N,H)),
\end{equation}
then we obtain 
\[
    \boldsymbol\Xi
    \in
    L^\infty\bigl(
        [0,\infty);
        \mathcal L_2(\widetilde U_0,\boldsymbol H)
    \bigr).
\]
Thus,  the assumptions underlying Proposition~\ref{prop:wp} are verified on the product Gelfand triple. In particular, the global $H$-Lipschitz property of $F$ and the boundedness of $\clL_N\clC_N$ yield the growth, monotonicity and coercivity estimates required for the coupled drift. Consequently, \eqref{eq:coupled-state-observer} admits a unique global strong variational solution $(X,\widehat X)$.

We now consider the observer error $\mathcal E:=\widehat X-X$. Subtracting \eqref{eq:obs-nominal} from \eqref{eq:observer-substituted} gives
\begin{equation}\label{eq:observer-error-formal}
\begin{split}
    \mathrm d\mathcal E
    +\bigl(
        A\mathcal E
        +A_{\rm rc}(t)\mathcal E
        &+F(t,\widehat X)-F(t,X)
        -\clL_N\clC_N\mathcal E
    \bigr)\,\mathrm dt \\&
    =
    -\clL_NR_N(t)\,\mathrm dW^{\rm obs}
    -\Xi(t)\,\mathrm dW .
\end{split}
\end{equation}
Following the oblique-projection construction used in~\cite[Sect.~1.2, Eq.~(1.20)]{Rod21-aut}, we choose $\clL_N=\clL_N(\lambda)$ so that
\begin{equation}\label{eq:obs-LC=BK}
   \clL_N\clC_N
    :=
    -\lambda\,
    P_{\mathcal{U}_L}^{\widetilde{\mathcal{U}}_L^\perp}
    P_{\widetilde{\mathcal{U}}_L}^{\mathcal{U}_L^\perp},
    \qquad
    \lambda>0 .
\end{equation}
Thus, the feedback forcing operator used in Theorem~\ref{thm:main} is interpreted here as an output-injection mechanism for the observer error. Since the true state $X$ is stochastic, the nonlinear term in the error equation is random. Now, we consider
\[
    \clN_X(t,\omega,y)
    :=
    F(t,X(t,\omega)+y)-F(t,X(t,\omega)),
    \qquad
    y\in V.
\]
Since $X$ is progressively measurable and $F$ is Borel measurable, $\clN_X$ is progressively measurable in $(t,\omega)$ for every $y\in V$. Then \eqref{eq:observer-error-formal} becomes
\begin{equation}\label{eq:observer-error}
\begin{split}
    \mathrm d\mathcal E
    +\bigl(
        A\mathcal E
        +A_{\rm rc}(t)\mathcal E
        &+\clN_X(t,\cdot,\mathcal E)
        +\lambda
        P_{\mathcal{U}_L}^{\widetilde{\mathcal{U}}_L^\perp}
        P_{\widetilde{\mathcal{U}}_L}^{\mathcal{U}_L^\perp}\mathcal E
    \bigr)\,\mathrm dt
    \\&=
    -\clL_NR_N(t)\,\mathrm dW^{\rm obs}
    -\Xi(t)\,\mathrm dW .
\end{split}
\end{equation}
Thus, \eqref{eq:observer-error} is not a direct application of Theorem~\ref{thm:main}, since $\clN_X$ is not a deterministic nonlinearity. However, the well-posedness of the coupled system yields an adapted error process with the required variational regularity. Moreover, we have 
\[
    \clN_X(t,\omega,0)=0
\]
and, for a.e. $t\in[0,\infty)$, every $\omega\in\Omega$, and all $y,z\in V$, we can write
\[
\begin{aligned}
    \|\clN_X(t,\omega,y)-\clN_X(t,\omega,z)\|_H
    &=
    \|F(t,X(t,\omega)+y)-F(t,X(t,\omega)+z)\|_H \\
    &\leq
    L_F\|y-z\|_H.
\end{aligned}
\]
Hence $\clN_X(t,\omega,\cdot)$ has the same global $H$-Lipschitz constant as $F(t,\cdot)$, uniformly in $(t,\omega)$. The stochastic right-hand side in \eqref{eq:observer-error} can be written as a single additive noise term driven by $\widetilde W$. Indeed, for 
\[
    \widetilde\Xi_{\mathcal E}(t)(u,\eta)
    :=
    -\Xi(t)u-\clL_NR_N(t)\eta,
    \qquad
    (u,\eta)\in\widetilde U_0,
\]
due to \eqref{eq:bounded_coupled_noise}, we can infer that 
\[
    \widetilde\Xi_{\mathcal E}
    \in
    L^\infty\bigl(
        [0,\infty);
        \mathcal L_2(\widetilde U_0,H)
    \bigr).
\]
Thus, the additive stochastic forcing satisfies the analogue of Assumption~\ref{ass:Xi} on the product noise space. Together with the adaptedness and variational regularity inherited from the
coupled system, the preceding Lipschitz estimate allows us to carry out the energy argument used in Theorem~\ref{thm:main} for the observer-error equation. If the coercivity condition generated by $\clL_N\clC_N$ is chosen as in the feedback stabilization result, this yields stabilization of the observer error up to a noise-dependent residual level.

\begin{remark}
The choice of parameter $\lambda$ leads to a natural trade-off. Increasing $\lambda$ strengthens the deterministic damping generated by $\clL_N\clC_N$, but it can also amplify the measurement noise through the coefficient $\clL_NR_N$. Thus, in the noisy case, the residual error level depends on the balance between stabilization and noise amplification. This effect would be interesting to investigate further, in particular through numerical simulations.
\end{remark}

\begin{remark}
The observer formulation above is restricted to additive noise in the state equation. If the state equation contains a multiplicative noise term $Z(t,X(t))\,\mathrm dW$ and the driving $Q$-Wiener process $W$ is not available to the observer, then this term appears in the error equation as an additional stochastic forcing term. In particular, the error equation does not contain the error-dependent term $\bigl(Z(t,\widehat X)-Z(t,X)\bigr)\,\mathrm dW$. Hence, multiplicative noise is not directly covered by the present implementable observer formulation. 

On the other hand, consider a common-noise formulation in which the state and the estimator are driven by the same $Q$-Wiener process $W$, and the estimator contains the multiplicative noise term $Z(t,\widehat X)\,\mathrm dW$. Then the error equation contains the error-dependent stochastic term  $\bigl(Z(t,\widehat X)-Z(t,X)\bigr)\,\mathrm dW$. This setting is closer to stochastic synchronization or common-noise data assimilation than to a standard filtering observer. In this case, Assumption~\ref{ass:standing}-\ref{ass:Z} provides the required control of the multiplicative stochastic term, and the energy argument of Theorem~\ref{thm:main} can be carried out under the corresponding coercivity condition. If, in addition, $\Xi\equiv0$ and $\clL_NR_N\equiv0$, one obtains the corresponding mean-square exponential stability and almost-sure logarithmic stability estimates for the observer error. If either additive term is nonzero, one instead obtains stabilization up to a noise-dependent residual level.
\end{remark}

\section{Numerical Scheme}\label{sec:numerical-scheme}

This section describes the discretization used in the numerical experiments of Section~\ref{sec:simulations}. We first describe the finite-dimensional truncation of the stochastic forcing, then introduce the spatial finite-element approximation, and finally specify the time-stepping scheme.

Throughout Sections~\ref{sec:numerical-scheme}--\ref{sec:simulations}, we consider system~\eqref{eq:sys-X} in the following concrete setting:

\begin{enumerate}[label=(\roman*)]
    \item \emph{Temporal domain.} Finite time interval $[0,T]$ for some $T>0$.
    \item \emph{Spatial domain.} $\mathcal O = (0,L_{x_1})\times(0,L_{x_2})$ is a bounded rectangle.
    \item \emph{Boundary conditions.} Homogeneous Neumann type ($\mathfrak T = \partial_n$).
    \item \emph{Noise.} $W$ is a trace-class $L^2(\mathcal{O})$-Wiener process on $H$, with
    covariance eigenpairs as specified in~\eqref{eq:Q-eigenpairs}.
    \item \emph{Feedback.} $K = K_L^{[\lambda]}$ is the oblique-projection based feedback operator
    constructed in \eqref{eq:KL-def}, with gain $\lambda > 0$.
\end{enumerate}
The fully discrete state $\bigl(X_i^{\mathcal J,h}\bigr)_{i=0}^{m} \subset V_h$ is produced by three successive discretization layers, indexed by the noise truncation index set $\mathcal J\subset \mathbb N_0^2$, the number of time steps $m\in\mathbb N_+$, and the spatial mesh size $h>0$:
\begin{equation}\label{eq:num-pipeline}
    W_t
    \;\xrightarrow{\;\textup{KL truncation}\;}\;
    W_t^{\mathcal J}
    \;\xrightarrow{\;\textup{Galerkin}\;}\;
    X^{\mathcal J,h}(t) \in V_h
    \;\xrightarrow{\;\textup{BEM}\;}\;
    \bigl( X_i^{\mathcal J,h} \bigr)_{i=0}^m \subset V_h
\end{equation}

The three layers are described in turn in Section~\ref{subsec:KL-truncation}--\ref{subsec:BEM}. A reference implementation of the scheme described below is publicly archived
on Zenodo~\cite{vonderHeydtRodrigues2026code}.

\subsection{Karhunen--Lo\`eve Truncation of the Wiener Process}\label{subsec:KL-truncation}

We begin with the first layer of the approximation pipeline~\eqref{eq:num-pipeline}, namely the replacement of the infinite-dimensional Wiener process $W$ by a finite-mode approximation. Recall from~\eqref{eq:W-repr} that
\begin{equation}\label{eq:W-series-num}
    W(t) = \sum_{k=1}^{\infty} \sqrt{\lambda_k}\,\beta_k(t)\,e_k ,
\end{equation}
where $(e_k)_{k\geq 1}$ is an orthonormal basis of $U$ consisting of eigenvectors of the covariance operator $Q$, with corresponding eigenvalues $(\lambda_k)_{k\geq 1}$, and $(\beta_k)_{k\geq 1}$ is a sequence of independent standard Brownian motions. In the concrete two-dimensional setting of Section~\ref{sec:numerical-scheme}--\ref{sec:simulations}, the covariance operator $Q\in\mathcal L(H)$ is defined by its eigenpairs
\begin{equation}\label{eq:Q-eigenpairs}
        \lambda_{j,k}
        = \Bigl(\Bigl(\frac{j\pi}{L_{x_1}}\Bigr)^2
               \!\!\!+\Bigl(\frac{k\pi}{L_{x_2}}\Bigr)^2
               \!\!\!+\ell^{-2}\Bigr)^{-\alpha},
        \ \
        e_{j,k}(x_1,x_2)
        = \frac{c_m c_n}{\sqrt{L_{x_1} L_{x_2}}}\,
          \cos\Bigl(\frac{j\pi}{L_{x_1}}x\Bigr)\cos\Bigl(\frac{k\pi}{L_{x_2}}y\Bigr),
\end{equation}
for $(j,k)\in\mathbb N_0^2$, where $c_0\coloneqq 1$ and $c_\ell\coloneqq\sqrt{2}$ for $\ell\geq 1$, so that $(e_{j,k})_{(j,k)\in\mathbb N_0^2}$ is an $L^2(\mathcal O)$-orthonormal basis. Equivalently, $Q = (-\Delta_{\mathrm{Neu}} + \ell^{-2})^{-\alpha}$ is the spectral $\alpha$-th inverse power of the shifted Neumann Laplacian $-\Delta_{\mathrm{Neu}} + \ell^{-2}$ on $\mathcal O$. For $\alpha>1$, Proposition~\ref{prop:noise-trunc-error} below shows that $Q$ is trace class. Moreover, we can write 
\[
    \|e_{j,k}\|_{L^\infty(\mathcal O)}^2
    = \frac{c_j^2 c_k^2}{L_{x_1}L_{x_2}}
    \leq \frac{4}{L_{x_1}L_{x_2}}
    \qquad \forall\,(j,k)\in\mathbb N_0^2,
\]
and therefore
\[
    \sum_{(j,k)\in\mathbb N_0^2} \lambda_{j,k}\|e_{j,k}\|_{L^\infty(\mathcal O)}^2
    \leq
    \frac{4}{L_{x_1}L_{x_2}}
    \sum_{(j,k)\in\mathbb N_0^2} \lambda_{j,k}.
\]
Hence \eqref{eq:sufficient-Q-cond} holds whenever $\alpha>1$, and thus so does the pointwise trace condition~\eqref{eq:Q-cond}. The parameters $\alpha>0$ and $\ell>0$ governing the covariance decay and noise intensity are specified in Section~\ref{sec:simulations}.

For truncation parameters $J_{x_1}, J_{x_2}\in\mathbb N_+$, we retain the modes indexed by the finite set
\begin{equation}\label{eq:J-index-set}
    \mathcal{J}
    \;\coloneqq\;
    \{0,\ldots,J_{x_1}-1\}\times\{0,\ldots,J_{x_2}-1\}
\end{equation}
and approximate $W$ by the $\mathcal J$-mode truncation~\cite[\S 10.2, Eq.~(10.12)]{lord2014introduction}.
\begin{equation}\label{eq:W-trunc-num}
    W^{\mathcal J}(t)
    \;\coloneqq\;
    \sum_{(j,k)\in\mathcal{J}}
    \sqrt{\lambda_{j,k}}\,\beta_{j,k}(t)\,e_{j,k}
    \;\in\;
    \operatorname{span}\bigl\{e_{j,k}
        \bigm| (j,k)\in\mathcal{J}\bigr\}
    \;\subset\; U.
\end{equation}

We define the $\mathcal J$-mode truncated stochastic partial differential equation by replacing $W$ with $W^{\mathcal J}$ in the closed-loop system~\eqref{eq:sys-X}. More precisely, for fixed $\mathcal J\in\mathbb N_0^2$, we consider
\begin{equation}\label{eq:sys-X-J}
    \begin{split}
        \mathrm dX^{\mathcal J}
        &+ \bigl(-\nu \Delta X^{\mathcal J} + a(t,x)X^{\mathcal J} + b(t,x) \cdot \nabla X^J  + f(X^{\mathcal J})\bigr)\,\mathrm dt \\&
        = g(t,x,X^{\mathcal J})\,\mathrm dW^{\mathcal J}
        + U_L^\diamond K_L^{[\lambda]} X^{\mathcal J}\,\mathrm dt, \\
        & \partial_n X\rest{\partial\mathcal O}=0,
    \qquad  X^{\mathcal J}\rest{t=0} = X_0 .    
    \end{split}
\end{equation}

It is convenient to make the corresponding stochastic integral explicit. For every $t\geq 0$, the diffusion term in~\eqref{eq:sys-X-J} can be written as
\begin{equation}\label{eq:stoch-int-J}
    \int_0^t g\bigl(s,\cdot,X^{\mathcal J}(s)\bigr)\,\mathrm dW^{\mathcal J}(s)
    \;=\;
    \sum_{(j,k)\in\mathcal{J}}
    \sqrt{\lambda_{j,k}}
    \int_0^t g\bigl(s,\cdot,X^{\mathcal J}(s)\bigr)\,e_{j,k}\,
    \mathrm d\beta_{j,k}(s).
\end{equation}
The following proposition bounds the mean-square error $\mathbb E\|W(t)-W^{\mathcal J}(t)\|_H^2$ for the specific covariance structure~\eqref{eq:Q-eigenpairs}. Its role within the full approximation pipeline~\eqref{eq:num-pipeline}, and the scope of the resulting numerical error discussion, are summarized in Remark~\ref{rem:convergence-informal}.

To simplify the trace and truncation estimates below, we introduce the constants
\begin{equation}\label{eq:KL-Balpha}
    B_\alpha
    \;:=\;
    \int_0^{\pi/2}\cos^{2\alpha-2}\phi\,\mathrm d\phi
    \;=\;
    \frac12 \mathrm{B}\!\left(\frac12,\alpha-\frac12\right)
    \;=\;
    \frac{\sqrt{\pi}\,\Gamma\!\left(\alpha-\frac12\right)}{2\,\Gamma(\alpha)},
    \qquad \alpha>\frac12,
\end{equation}
and
\begin{equation}\label{eq:KL-Balpha-half}
    B_{\alpha-\frac12}
    \;:=\;
    \int_0^{\pi/2}\cos^{2\alpha-3}\phi\,\mathrm d\phi
    \;=\;
    \frac12 \mathrm{B}\!\left(\frac12,\alpha-1\right)
    \;=\;
    \frac{\sqrt{\pi}\,\Gamma(\alpha-1)}{2\,\Gamma\!\left(\alpha-\frac12\right)},
    \qquad \alpha>1,
\end{equation}
where the identities follow from \cite[\href{https://dlmf.nist.gov/5.12.E2}{eq.~(5.12.2)}]{NIST:DLMF}.

\begin{proposition}[KL truncation error]\label{prop:noise-trunc-error}
    Let $\alpha>0$, $\ell>0$, $L_{x_1}, L_{x_2}>0$, and let $(e_{j,k},\lambda_{j,k})$ be as defined in~\eqref{eq:Q-eigenpairs}. Furthermore, for truncation parameters $J_{x_1}, J_{x_2}\in\mathbb N_+$ with $J_{x_1},J_{x_2}\geq 2$, let $\mathcal{J}$ be as in~\eqref{eq:J-index-set}, and let $W^{\mathcal J}(t)$ be the $\mathcal J$-mode truncation~\eqref{eq:W-trunc-num} of the associated $Q$-Wiener process $W(t)$ in $H$. Then the following statement hold:
    \begin{enumerate}[label=(\roman*)]
        \item \emph{Trace-class condition.} $Q$ is trace-class if and only if $\alpha > 1$, and 
        \begin{equation}\label{eq:trace-Q}
            \mathrm{Tr}(Q) = \sum_{(j,k) \in \mathbb{N}_0^2} \lambda_{j,k}
            \,\leq\,
            \ell^{2\alpha} + \frac{(L_{x_1}+L_{x_2})\ell^{2\alpha-1}}{\pi}B_\alpha + \frac{L_{x_1}L_{x_2}\ell^{2\alpha-2}}{\pi^2}B_\alpha B_{\alpha-\frac12},
        \end{equation}
        where $B_\alpha$ and $B_{\alpha-\frac{1}{2}}$ are as defined in~\eqref{eq:KL-Balpha}--\eqref{eq:KL-Balpha-half}.
        \item \emph{Truncation error.} 
        For $\alpha > 1$, we have
        \begin{equation}\label{eq:trunc-error}
            \mathbb{E}\|W(t) - W^{\mathcal J}(t)\|_H^2 = t \sum_{\mathclap{(j,k) \notin \mathcal{J}}} \lambda_{j,k} \;\leq\; C_{\alpha,L_{x_1},L_{x_2}} \cdot t \cdot (\hat J - 1)^{-(2\alpha - 2)},
        \end{equation}
        where $\hat J \coloneqq \min(J_{x_1}, J_{x_2})$ and the constant is given explicitly by
        \begin{equation}\label{eq:trunc-const}
            C_{\alpha,L_{x_1},L_{x_2}} \;:=\; \frac{B_\alpha \,(L_{x_1}^{2\alpha-1} L_{x_2} + L_{x_1} L_{x_2}^{2\alpha-1})} {(\alpha-1) \, \pi^{2\alpha-1}} +  \frac{L_{x_1}^{2\alpha} + L_{x_2}^{2\alpha}} {(2\alpha-1) \, \pi^{2\alpha}}
        \end{equation}
        In particular, $\mathbb{E}\|W(t) - W^{\mathcal J}(t)\|_H^2 \to 0$ as
        $J_{x_1},J_{x_2}\to\infty$ for any $\alpha > 1$.
    \end{enumerate}
\end{proposition}
\begin{proof}
We prove the two statements in sequence.

\medskip
\noindent\textbf{Proof of (i): trace-class property.}
The family $(e_{j,k})_{(j,k)\in\mathbb N_0^2}$ is the tensor-product cosine basis of $H=L^2(\mathcal O)$, hence an orthonormal basis of eigenvectors of $Q$, with eigenvalues $\lambda_{j,k}$ given by~\eqref{eq:Q-eigenpairs}. Therefore, we have
\begin{equation}\label{eq:TrQ-sum}
    \operatorname{Tr}(Q)
    =
    \sum_{(j,k)\in\mathbb N_0^2}\lambda_{j,k}
    =
    \sum_{(j,k)\in\mathbb N_0^2}
    \Bigl(
        \frac{\pi^2 j^2}{L_{x_1}^2}
        +
        \frac{\pi^2 k^2}{L_{x_2}^2}
        +
        \ell^{-2}
    \Bigr)^{-\alpha}.
\end{equation}
We denote 
\[
    f(s,t)
    :=
    \Bigl(
        \frac{\pi^2 s^2}{L_{x_1}^2}
        +
        \frac{\pi^2 t^2}{L_{x_2}^2}
        +
        \ell^{-2}
    \Bigr)^{-\alpha},
    \qquad s,t\ge 0.
\]
Then, since $f$ is decreasing in each variable, we can write
\begin{equation}\label{eq:interior-int-comp-TrQ}
    \sum_{j\ge1,k\ge1}\lambda_{j,k}
    \le
    \int_0^\infty\!\!\int_0^\infty f(s,t)\,\mathrm dt\,\mathrm ds,
\end{equation}
and
\begin{equation}\label{eq:boundary-int-comp-TrQ}
    \sum_{j\ge1}\lambda_{j,0}\le \int_0^\infty f(s,0)\,\mathrm ds,
    \qquad
    \sum_{k\ge1}\lambda_{0,k}\le \int_0^\infty f(0,t)\,\mathrm dt.
\end{equation}
For fixed $s\ge0$, let
\[
    A(s):=\frac{\pi^2 s^2}{L_{x_1}^2}+\ell^{-2}.
\]
Using the substitution
\[
    t=\frac{L_{x_2}}{\pi}A(s)^{1/2}\tan\phi,
    \qquad \phi\in[0,\pi/2),
\]
and the definition of $B_\alpha$ in~\eqref{eq:KL-Balpha}, we obtain
\begin{equation}\label{eq:t-integral}
    \int_0^\infty
    \Bigl(A(s)+\frac{\pi^2 t^2}{L_{x_2}^2}\Bigr)^{-\alpha}\,\mathrm dt
    =
    \frac{L_{x_2}}{\pi}A(s)^{\frac12-\alpha}
    \int_0^{\pi/2}\cos^{2\alpha-2}\phi\,\mathrm d\phi
    =
    \frac{L_{x_2}B_\alpha}{\pi}A(s)^{\frac12-\alpha}.
\end{equation}
Similarly, the substitution
\[
    s=\frac{L_{x_1}}{\pi\ell}\tan\psi,
    \qquad \psi\in[0,\pi/2),
\]
and the definition of $B_{\alpha-\frac12}$ in~\eqref{eq:KL-Balpha-half}, yields
\[
    \int_0^\infty A(s)^{\frac12-\alpha}\,\mathrm ds
    =
    \frac{L_{x_1}\ell^{2\alpha-2}}{\pi}
    \int_0^{\pi/2}\cos^{2\alpha-3}\psi\,\mathrm d\psi
    =
    \frac{L_{x_1}\ell^{2\alpha-2}}{\pi}B_{\alpha-\frac12}.
\]
Hence, for $\alpha>1$, we have 
\begin{equation}\label{eq:interior-TrQ-bound-correct}
    \sum_{j\ge1,k\ge1}\lambda_{j,k}
    \le
    \frac{L_{x_1}L_{x_2}\ell^{2\alpha-2}}{\pi^2}
    B_\alpha B_{\alpha-\frac12}.
\end{equation}
Moreover, evaluating the one-dimensional integrals in~\eqref{eq:boundary-int-comp-TrQ} in the same way gives
\begin{equation}\label{eq:boundary-TrQ-bound-correct}
    \sum_{j\ge1}\lambda_{j,0}
    \le
    \frac{L_{x_1}\ell^{2\alpha-1}}{\pi}B_\alpha,
    \qquad
    \sum_{k\ge1}\lambda_{0,k}
    \le
    \frac{L_{x_2}\ell^{2\alpha-1}}{\pi}B_\alpha.
\end{equation}
Since $\lambda_{0,0}=\ell^{2\alpha}$, we conclude for $\alpha>1$ that 
\begin{equation}\label{eq:trace-Q-correct}
    \Tr(Q)
    \le
    \ell^{2\alpha}
    +
    \frac{(L_{x_1}+L_{x_2})\ell^{2\alpha-1}}{\pi}B_\alpha
    +
    \frac{L_{x_1}L_{x_2}\ell^{2\alpha-2}}{\pi^2}
    B_\alpha B_{\alpha-\frac12}
    <\infty.
\end{equation}
It remains now to show that the series in~\eqref{eq:TrQ-sum} diverges for $\alpha\le1$. Since $f$ is decreasing in each variable, for every $(j,k)\in\mathbb N_0^2$ and every $(s,t)\in[j,j+1]\times[k,k+1]$, we have
\[
    f(s,t)\le f(j,k)=\lambda_{j,k}.
\]
Hence, it holds that 
\[
    \sum_{(j,k)\in\mathbb N_0^2}\lambda_{j,k}
    \ge
    \int_0^\infty\!\!\int_0^\infty f(s,t)\,\mathrm dt\,\mathrm ds.
\]
After the scaling $x=\frac{\pi s}{L_{x_1}}$ and $y=\frac{\pi t}{L_{x_2}}$, we obtain the exact identity
\[
    \int_0^\infty\!\!\int_0^\infty f(s,t)\,\mathrm dt\,\mathrm ds
    =
    \frac{L_{x_1}L_{x_2}}{\pi^2}
    \int_0^\infty\!\!\int_0^\infty
    (x^2+y^2+\ell^{-2})^{-\alpha}\,\mathrm dy\,\mathrm dx.
\]
Since $(x^2+y^2+\ell^{-2})^{-\alpha}\asymp (x^2+y^2)^{-\alpha}$ as $x^2+y^2\to\infty$, there exist $R>0$ and $c>0$ such that
\[
    (x^2+y^2+\ell^{-2})^{-\alpha}
    \ge
    c\,(x^2+y^2)^{-\alpha},
\]
whenever $x^2+y^2\ge R^2$. Passing to polar coordinates in the first quadrant, we obtain
\[
    \int_0^\infty\!\!\int_0^\infty (x^2+y^2+\ell^{-2})^{-\alpha}\,\mathrm dy\,\mathrm dx
    \ge
    c\int_0^{\pi/2}\!\!\int_R^\infty r^{1-2\alpha}\,\mathrm dr\,\mathrm d\theta.
\]
The last integral diverges for $\alpha\le1$, because $r^p$ is not integrable at infinity when $p\ge -1$ and here $p=1-2\alpha$. Thus $\sum_{(j,k)\in\mathbb N_0^2}\lambda_{j,k}=\infty$ for $\alpha\le1$.

\medskip
\noindent\textbf{Proof of (ii): truncation estimate.}
For the remainder of the proof, we assume $\alpha>1$ and set $\hat J:=\min(J_{x_1},J_{x_2})$ and $\mathcal J$ as in \eqref{eq:J-index-set}. We recall that
\[
    f(s,t)=
    \Bigl(
        \frac{\pi^2 s^2}{L_{x_1}^2}
        +
        \frac{\pi^2 t^2}{L_{x_2}^2}
        +
        \ell^{-2}
    \Bigr)^{-\alpha},
    \qquad
    A(s)=\frac{\pi^2 s^2}{L_{x_1}^2}+\ell^{-2}.
\]
By (i), $Q$ is trace class, so $W_t$ is a well-defined $H$-valued $Q$-Wiener process, and it satisfies 
\[
    W_t-W_t^{\mathcal J}
    =
    \sum_{(j,k)\notin\mathcal J}
    \sqrt{\lambda_{j,k}}\,\beta_{j,k}(t)\,e_{j,k}.
\]
Using the orthonormality of $(e_{j,k})_{(j,k)\in\mathbb N_0^2}$ in $H$ and the independence of the Brownian motions, we obtain
\begin{equation}\label{eq:ito-isom}
    \mathbb E\|W_t-W_t^{\mathcal J}\|_H^2
    =
    t\sum_{(j,k)\notin\mathcal J}\lambda_{j,k}.
\end{equation}
Moreover, we can write 
\begin{equation}\label{eq:S1S2}
    \sum_{(j,k)\notin\mathcal J}\lambda_{j,k}
    =
    \underbrace{\sum_{j\ge J_{x_1}}\sum_{k\ge0}\lambda_{j,k}}_{=:S_1}
    +
    \underbrace{\sum_{j=0}^{J_{x_1}-1}\sum_{k\ge J_{x_2}}\lambda_{j,k}}_{=:S_2}.
\end{equation}
Since $f$ is decreasing in each variable, the comparison argument from the proof of (i) yields
\[
    S_1
    \le
    \int_{J_{x_1}-1}^{\infty}\int_0^{\infty} f(s,t)\,\mathrm dt\,\mathrm ds
    +
    \int_{J_{x_1}-1}^{\infty} f(s,0)\,\mathrm ds.
\]
By \eqref{eq:t-integral} and the bound $A(s)\ge \pi^2s^2/L_{x_1}^2$, we have
\[
    \int_{J_{x_1}-1}^{\infty}\int_0^{\infty} f(s,t)\,\mathrm dt\,\mathrm ds
    \le
    \frac{L_{x_2}B_\alpha}{\pi}
    \int_{J_{x_1}-1}^{\infty}A(s)^{\frac12-\alpha}\,\mathrm ds
    \le
    \frac{L_{x_1}^{2\alpha-1}L_{x_2}B_\alpha}
         {2\pi^{2\alpha-1}(\alpha-1)}
    (J_{x_1}-1)^{2-2\alpha},
\]
and
\[
    \int_{J_{x_1}-1}^{\infty} f(s,0)\,\mathrm ds
    \le
    \int_{J_{x_1}-1}^{\infty}
    \Bigl(\frac{\pi s}{L_{x_1}}\Bigr)^{-2\alpha}\,\mathrm ds
    =
    \frac{L_{x_1}^{2\alpha}}
         {\pi^{2\alpha}(2\alpha-1)}
    (J_{x_1}-1)^{1-2\alpha}.
\]
Since $J_{x_1}\ge2$ and $1-2\alpha<2-2\alpha<0$, it follows that
\[
    (J_{x_1}-1)^{1-2\alpha}
    \le
    (J_{x_1}-1)^{2-2\alpha},
\]
and, therefore,
\[
    S_1
    \le
    \Bigl(
        \frac{L_{x_1}^{2\alpha-1}L_{x_2}B_\alpha}
             {2\pi^{2\alpha-1}(\alpha-1)}
        +
        \frac{L_{x_1}^{2\alpha}}
             {\pi^{2\alpha}(2\alpha-1)}
    \Bigr)
    (J_{x_1}-1)^{2-2\alpha}.
\]
By nonnegativity of the eigenvalues, we obtain
\[
    S_2
    =
    \sum_{j=0}^{J_{x_1}-1}\sum_{k\ge J_{x_2}}\lambda_{j,k}
    \le
    \sum_{k\ge J_{x_2}}\sum_{j\ge0}\lambda_{j,k}.
\]
The right-hand side has the same form as $S_1$, with the roles of $(j,L_{x_1},J_{x_1})$ and $(k,L_{x_2},J_{x_2})$ interchanged. Hence the preceding argument yields
\[
    S_2
    \le
    \Bigl(
        \frac{L_{x_1}L_{x_2}^{2\alpha-1}B_\alpha}
             {2\pi^{2\alpha-1}(\alpha-1)}
        +
        \frac{L_{x_2}^{2\alpha}}
             {\pi^{2\alpha}(2\alpha-1)}
    \Bigr)
    (J_{x_2}-1)^{2-2\alpha}.
\]
Since $\hat J=\min(J_{x_1},J_{x_2})$ and $2-2\alpha<0$, we have
\[
    (J_{x_1}-1)^{2-2\alpha}
    \le
    (\hat J-1)^{2-2\alpha},
    \qquad
    (J_{x_2}-1)^{2-2\alpha}
    \le
    (\hat J-1)^{2-2\alpha}.
\]
Combining the bounds for $S_1$ and $S_2$ with \eqref{eq:S1S2}, and then using \eqref{eq:ito-isom}, yields \eqref{eq:trunc-error} with the constant \eqref{eq:trunc-const}. In particular, we can infer that
\[
    \mathbb E\|W_t-W_t^{\mathcal J}\|_H^2 \to 0
    \qquad\text{as}\qquad
    J_{x_1},J_{x_2}\to\infty .
\]
This completes the proof.
\end{proof}

\begin{remark}[Scope of the truncation estimate]\label{rem:convergence-informal}
Proposition~\ref{prop:noise-trunc-error} provides a rigorous mean-square error bound for the first approximation layer in~\eqref{eq:num-pipeline}, namely the Karhunen--Lo\`eve truncation of the Wiener process.

After this truncation, the closed-loop equation~\eqref{eq:sys-X-J} is driven by finitely many independent Brownian motions. This suggests a possible route toward a fully discrete error analysis through abstract space-time approximation results for stochastic evolution equations with finite-dimensional noise, such as those in~\cite{gyongy_millet2009}.

We do not carry out that program here. A rigorous application of such results to the present truncated closed-loop problem would require additional regularity and consistency assumptions, as well as a verification of the corresponding hypotheses for our finite-element and backward-Euler discretizations.

Accordingly, the purpose of Proposition~\ref{prop:noise-trunc-error} is to justify the approximation of the stochastic forcing used in the simulations. The convergence of the full three-layer scheme is included only as numerical support for the implementation, and the experiments in Section~\ref{sec:simulations} show that the truncation rate is captured accurately and that the associated state error decays even faster in the reported examples.
\end{remark}

\subsection{Spatial Discretization}\label{subsec:galerkin}

We cover $\mathcal O=[0,L_{x_1}]\times[0,L_{x_2}]$ with a structured rectangular mesh $\mathcal{T}_h$ of $n_{x_1}\times n_{x_2}$ uniform quadrilateral elements. This choice aligns the discretization with the rectangular actuator supports and with the tensor-product structure of the Neumann-Laplacian eigenbasis~\eqref{eq:Q-eigenpairs}, which simplifies both the assembly and the evaluation of the truncated noise. It is also natural in the present setting, since the domain $\mathcal O$ itself is rectangular. We take $V_h \subset H^1\mathcal O)$ to be the standard bilinear $Q_1$ finite element space, with $n \coloneqq(n_{x_1}+1)(n_{x_2}+1)$ degrees of freedom and global nodal basis ${\phi_1,\ldots,\phi_n}$.The following matrices are assembled from the mesh and problem data. The \emph{mass matrix} $M\in\mathbb R^{n\times n}$ and the \emph{diffusion stiffness matrix} $K_\nu\in\mathbb R^{n\times n}$ are time-independent:
\begin{equation}\label{eq:M-Knu}
    M_{ij} \;=\; \int_\mathcal O \phi_i\,\phi_j\,\mathrm d x,
    \qquad
    (K_\nu)_{ij} \;=\; \nu\int_\mathcal O \nabla\phi_i \cdot \nabla\phi_j\,\mathrm d x.
\end{equation}
The \emph{reaction--convection matrix} $K_\mathrm{rc}(t)\in\mathbb R^{n\times n}$ is assembled at each time step:
\begin{equation}\label{eq:Krc}
    \bigl(K_\mathrm{rc}(t)\bigr)_{ij}
    \;=\;
    \int_\mathcal O a(t,x)\,\phi_i\,\phi_j\,\mathrm d x
    \;+\;
    \int_\mathcal O b(t,x)\cdot\nabla\phi_j\,\phi_i\,\mathrm d x.
\end{equation}
We write the combined drift matrix as
\begin{equation}\label{eq:K-full}
    K(t) \;\coloneqq\; K_\nu + K_\mathrm{rc}(t).
\end{equation}

The domain $\mathcal O$ is partitioned into $N = N_1 \times N_2$ congruent rectangular sub-cells, with $N_1$ columns in $x$ and $N_2$ rows in $y$. The center of sub-cell $(n_1,n_2)$ is
\begin{equation}\label{eq:actuator-centers}
    c_{n_1,n_2}
    \;\coloneqq\;
    \Bigl(\tfrac{(n_1-\tfrac{1}{2})L_{x_1}}{N_1},\;
          \tfrac{(n_2-\tfrac{1}{2})L_{x_2}}{N_2}\Bigr),
    \quad
    n_1=1,\ldots,N_1,\quad n_2=1,\ldots,N_2.
\end{equation}
Each actuator support $\clo_i \subset \mathcal O$ is the open axis-aligned box centered at $c_{n_1,n_2}$ with half-widths $(r_{x_1}, r_{x_2})$ chosen so that $|\clo_i| = r_{\mathrm{vol}}\cdot|\mathcal O|/N$, where $r_{\mathrm{vol}}\in(0,1]$ is the prescribed volume fraction parameter; see Figure~\ref{fig.suppActR}.

In the numerical implementation we take $N=L^2$. The actuator and auxiliary subspaces $\mathcal U_L$ and $\widetilde{\mathcal U}_L$ are represented on $V_h$ by the matrices
\begin{equation}\label{eq:BCtilde}
    B\in\mathbb R^{n\times N},\quad B_{ij}=\int_{\clo_j}\phi_i(x)\,\mathrm{d}x,
    \qquad
    \widetilde B\in\mathbb R^{n\times N},\quad \widetilde B_{ij}=\widetilde\phi_j(x_i),
\end{equation}
where $B$ is the actuator matrix (columns are the $M$-Riesz representatives of
$\1_{\clo_j}$) and $\widetilde B$ contains the nodal evaluations of the
$N$ auxiliary basis functions chosen so that the discrete direct-sum
condition~\eqref{eq:direct-sum} holds (concretely, $\widetilde\phi_j$ are
sin-smoothed bump functions supported on $\clo_j$, see Section~\ref{sec:simulations}).
The Galerkin realization of the oblique-projection feedback
$-\lambda\, P_{\mathcal{U}_L}^{\widetilde{\mathcal{U}}_L^\perp} \circ P_{\widetilde{\mathcal{U}}_L}^{\mathcal{U}_L^\perp}$ is then the $n\times n$ matrix
\begin{equation}\label{eq:Kfeed}
    K_{\mathrm{feed}}
    \coloneqq
    -\lambda\;
    B\,(\widetilde{B}^\top M B)^{-1}\widetilde{B}^\top M
    \;\cdot\;
    \widetilde{B}\,(B^\top M\widetilde{B})^{-1}B^\top M.
\end{equation}
The formula is the composition of two $M$-oblique projections: the right factor
$\widetilde{B}(B^\top M\widetilde{B})^{-1}B^\top M$ projects onto
$\operatorname{span}(\widetilde{B})$ along $\operatorname{span}(B)^{\perp_M}$,
and the left factor $B(\widetilde{B}^\top M B)^{-1}\widetilde{B}^\top M$
projects the result onto $\operatorname{span}(B)$ along $\operatorname{span}(\widetilde{B})^{\perp_M}$. The two $N\times N$ Gram matrices $\widetilde{B}^\top M B$ and $B^\top M\widetilde{B}=({\widetilde{B}^\top M B})^\top$ are each other's transpose and are the only objects that must be factored; this is done once offline, independently of time and the solution.

Projecting the truncated equation~\eqref{eq:sys-X-J} onto $V_h$ by the standard Galerkin procedure, the coefficient vector $\bm{X}^{\mathcal J,h}(t)\in\mathbb R^n$ of the semi-discrete solution $X^{\mathcal J,h}(t) = \sum_{i=1}^n [\bm{X}^{\mathcal J,h}(t)]_i\,\phi_i \in V_h$ satisfies the It\^o stochastic differential equation
\begin{equation}\label{eq:semi-discrete}
    \begin{split}
        M\,\mathrm d\bm{X}^{\mathcal J,h} &+ \bigl( K(t) -K_{\mathrm{feed}} \bigr) \bm{X}^{\mathcal J,h} \,\mathrm d t + M\, f\bigl(\bm{X}^{\mathcal J,h}\bigr) \,\mathrm d t \\ &\;=\; \sum_{(j,k)\in\mathcal{J}} \sqrt{\lambda_{j,k}}\, M \bigl( \bm{g}(t,\bm{X}^{\mathcal J,h}) \odot \bm{e}_{j,k} \bigr) \,\mathrm d\beta_{j,k},  \qquad \bm{X}^{\mathcal J,h}(0) = \bm{u}_0,
    \end{split}
\end{equation}
where $K(t) = K_\nu + K_{\mathrm{rc}}(t)$ is the combined drift matrix of~\eqref{eq:K-full}, $\odot$ denotes the Hadamard product, $\bm{g}(t,\bm{X}^{\mathcal J,h})\in\mathbb R^n$ is the nodal Nemytskii vector with $[\bm{g}(t,\bm{X}^{J\mathcal ,h})]_i \coloneqq g(t,x_i,[\bm{X}^{\mathcal J,h}]_i)$, $\bm{u}_0\in\mathbb R^n$ is the nodal interpolation of $X_0$, and $\bm{e}_{j,k}\in\mathbb R^n$ is the vector of pointwise nodal evaluations of the eigenfunction $e_{j,k}$,
\begin{equation}\label{eq:ejk-nodal}
    [\bm{e}_{j,k}]_i \;\coloneqq\; e_{j,k}(x_i) \;=\; \frac{c_j\,c_k}{\sqrt{L_{x_1} L_{x_2}}}\, \cos\Bigl(\frac{j\pi}{L_{x_1}}x_i\Bigr) \cos\Bigl(\frac{k\pi}{L_{x_2}}y_i\Bigr),
\end{equation}
with $c_0\coloneqq 1$, $c_j\coloneqq\sqrt{2}$ for $j\geq 1$
(cf.~\eqref{eq:Q-eigenpairs}).

\subsection{Time Discretization}\label{subsec:BEM}
Let $m \in \mathbb{N}_+$ and set $\tau \coloneqq T/m$. We partition $[0,T]$ uniformly into the grid points $t_i \coloneqq i\tau$ for $i = 0, \ldots, m$. The fully discrete solution is represented at each time step by its coefficient vector $\bm{X}_i^{\mathcal J,h} \in \mathbb{R}^n$, defined as the nodal coordinate vector of the fully discrete FE function $X_i^{\mathcal J,h} \in V_h$ with respect to the basis $\{\phi_1,\ldots,\phi_n\}$. We further define, for each $i = 0, \ldots, m-1$, the nodal diffusion vector
\begin{equation}\label{eq:nodal-diffusion-vec}
    \bm{g}_i \;\coloneqq\; \bigl( g\bigl(t_i, x_j,[\bm{X}_i^{\mathcal J,h}]_j \bigr) \bigr)_{j=1}^{n} \;\in\; \mathbb{R}^n,
\end{equation}
where $x_j$ denotes the $j$-th FE node and $[\bm{X}_i^{\mathcal J,h}]_j$ its $j$-th component, and the KL-truncated nodal Wiener increment
\begin{equation}\label{eq:nodal-wiener-inc}
    \Delta\bm{w}_i \;\coloneqq\; \sum_{(j,k)\in\mathcal{J}}
    \sqrt{\lambda_{j,k}}\,\bm{e}_{j,k}\,\Delta\beta_{j,k}^i
    \;\in\; \mathbb{R}^n,
\end{equation}
where $\Delta\beta_{j,k}^i \coloneqq \beta_{j,k}(t_{i+1}) - \beta_{j,k}(t_i) \sim \mathcal{N}(0,\tau)$ are independent Gaussian increments and $[\bm{e}_{j,k}]_i \coloneqq e_{j,k}(x_i)$ is the vector of nodal evaluations of the $(j,k)$-th KL eigenfunction.

Given the time grid and the notation of the preceding paragraph, we discretize the semi-discrete system~\eqref{eq:semi-discrete} by the following \emph{implicit--explicit backward Euler--Maruyama} (IMEX-BEM) scheme: for each $i = 0, \ldots, m-1$, given $\bm{X}_i^{\mathcal J,h} \in \mathbb{R}^n$, determine $\bm{X}_{i+1}^{\mathcal J,h} \in \mathbb{R}^n$ by solving the linear system
\begin{equation}\label{eq:BEM-scheme}
    \begin{split}
        \bigl( M + \tau &K(t_{i+1}) - \tau K_{\mathrm{feed}} \bigr) \, \bm{X}_{i+1}^{\mathcal J,h} \;=\; M \, \bm{X}_i^{\mathcal J,h} - \tau \, M \, f\bigl( \bm{X}_i^{\mathcal J,h} \bigr) + M\bigl( \bm{g}_i \odot \Delta\bm{w}_i \bigr), \\
        &\bm{X}_0^{\mathcal J,h} = \bm{X}_0,
    \end{split}
\end{equation}
where every matrix coefficient was defined in Section~\ref{subsec:galerkin}, and $\bm{g}_i \in \mathbb{R}^n$, $\Delta\bm{w}_i \in \mathbb{R}^n$ are as introduced in~\eqref{eq:nodal-diffusion-vec}--\eqref{eq:nodal-wiener-inc}. The scheme~\eqref{eq:BEM-scheme} uses an \emph{implicit--explicit} (IMEX) splitting of the drift: the linear stiff part $K(t_{i+1}) - K_{\mathrm{feed}}$ is evaluated at the new time step $t_{i+1}$ and moved to the left-hand side, while the nonlinearity $f$ is evaluated explicitly at the known step $t_i$. The rationale for this splitting becomes clear at the level of the time-stepping, where the SPDE has already been reduced, via the KL truncation of Section~\ref{subsec:KL-truncation} and the Galerkin discretization of Section~\ref{subsec:galerkin}, to a finite-dimensional It\^o SODE in $\mathbb{R}^n$. The implicit treatment of the linear stiff drift and the explicit treatment of the nonlinearity are motivated by numerical stability. For nonlinear SODEs with a one-sided Lipschitz drift, the backward Euler--Maruyama method is exponentially mean-square stable for all step sizes $\tau > 0$. By contrast, the forward Euler--Maruyama method can diverge in mean square for any fixed step size~\cite[Lem.~4.1, Thm.~4.4, Cor.~4.5]{higham_mao_stuart2003}, and this failure becomes more severe as the spatial mesh is refined and the stiffness of $M^{-1}K(t)$ increases~\cite[Thms.~4.1,~4.3]{higham2000}.

\subsection{Algorithm Overview}\label{subsec:algorithm-overview}

We summarize the fully discrete scheme in two algorithms. Algorithm~\ref{alg:offline} collects the offline computations that depend only on the mesh, the actuator geometry, and the gain parameter~$\lambda$; these are performed once before the time loop. Algorithm~\ref{alg:online} describes the IMEX--BEM time-stepping loop that produces the discrete trajectory $\bigl(X_i^{\mathcal J,h}\bigr)_{i=0}^{m}$.

\begin{algorithm}[htbp]
\caption{Offline: Feedback operator $K_{\mathrm{feed}}$}
\label{alg:offline}
\begin{algorithmic}[1]
\Require Mesh $\mathcal{T}_h$, actuator supports $\{\clo_i\}_{i=1}^N$,
         auxiliary functions $\{\widetilde\phi_i\}_{i=1}^N$,
         gain parameter $\lambda > 0$.
\State Assemble the consistent mass matrix $M$
\State Assemble the actuator matrix $B \in \mathbb{R}^{n \times N}$ and the auxiliary matrix $\widetilde{B} \in \mathbb{R}^{n \times N}$ (cf.~\eqref{eq:BCtilde}).
\State Factorize the two $N\times N$ Gram matrices $\widetilde{B}^\top M B$ and $\widetilde{B}^\top M \widetilde{B}$.
\State Compute the feedback operator
       (cf.~\eqref{eq:Kfeed}):
       \[
           K_{\mathrm{feed}}
           \;\leftarrow\;
           -\lambda\;
            B\,(\widetilde{B}^\top M B)^{-1}\widetilde{B}^\top M
            \;\cdot\;
            \widetilde{B}\,(B^\top M\widetilde{B})^{-1}B^\top M.
       \]
\Ensure $M$,\; $K_{\mathrm{feed}}$.
\end{algorithmic}
\end{algorithm}

\begin{algorithm}[htbp]
\caption{Online: IMEX--BEM time-stepping}
\label{alg:online}
\begin{algorithmic}[1]
\Require Initial vector $\bm{X}_0^{\mathcal J,h}$, time grid $\{t_i\}_{i=0}^{m}$,
         matrices $M$, $K_{\mathrm{feed}}$ from Algorithm~\ref{alg:offline}.
\State \textbf{Pre-loop (once).}
       Build the KL noise basis matrix $\Phi \in \mathbb{R}^{n \times |\mathcal J|}$ and the scaled square-root eigenvalue vector $\bm{s} \in \mathbb{R}^{|\mathcal J|}$, $s_k = \sqrt{\lambda_k}\,c_{j}\,c_{l}/\!\sqrt{|\mathcal{O}|}$, for the $|\mathcal J|$ retained modes (cf.~\eqref{eq:nodal-wiener-inc}).
\For{$i = 0, 1, \ldots, m-1$}
    \State \textbf{Noise increment.}
           Draw $\bm{\xi} \sim \mathcal{N}(0, I_{|\mathcal J|})$ and set (cf.~\eqref{eq:nodal-wiener-inc})
           \[
               \Delta\bm{w}_i
               \;\leftarrow\;
               \sqrt{\tau}\;\Phi\,(\bm{s} \odot \bm{\xi}).
           \]

    \State \textbf{Diffusion coefficient.}
           Evaluate $\bm{g}_i \in \mathbb{R}^n$ nodewise:
           $[\bm{g}_i]_k = g(t_i,\, x_k,\, [\bm{X}_i^{\mathcal J,h}]_k)$
           \phantom{epty}(cf.~\eqref{eq:nodal-diffusion-vec}).
    \State \textbf{Noise right-hand side.}
           $\bm{n}_i \leftarrow M(\bm{g}_i \odot \Delta\bm{w}_i)$.
    \State \textbf{Stiffness assembly.}
           Assemble $K(t_{i+1}) = K_\nu + K_\mathrm{rc}(t_{i+1})$
           (cf.~\eqref{eq:K-full}).
    \State \textbf{System matrix and right-hand side.}
           Form (cf.~\eqref{eq:BEM-scheme})
           \begin{align*}
               S_i    &\;\leftarrow\; M + \tau\bigl(K(t_{i+1}) - K_{\mathrm{feed}}\bigr), \\
               \bm{r}_i &\;\leftarrow\; M \bm{X}_i^{\mathcal J,h}
                                        - \tau M f(\bm{X}_i^{\mathcal J,h})
                                        + \bm{n}_i.
           \end{align*}
    \State \textbf{Linear solve.}\label{alg:online-linsolve}
           Solve $S_i\, \bm{X}_{i+1}^{\mathcal J,h} = \bm{r}_i$.
           \hfill\Comment{one linear solve per step}
\EndFor
\Ensure Trajectory $\bigl(\bm{X}_i^{\mathcal J,h}\bigr)_{i=0}^{m}$.
\end{algorithmic}
\end{algorithm}

\section{Simulations}\label{sec:simulations}

We verify the theoretical results of Theorem~\ref{thm:main} and investigate the influence of actuator placement, number of actuators, noise intensity, and nonlinear effects on the closed-loop stabilization behavior. As described in Section \ref{sec:numerical-scheme}, we consider system~\eqref{eq:sys-X} with homogeneous Neumann boundary conditions. For the covariance operator $Q = (-\Delta_{\mathrm{Neu}} + \ell^{-2})^{-\alpha}$, we set $\alpha = 1.5$ and $\ell = 0.25$. The parameter $\sigma\geq 0$ appears explicitly as a global scaling in front of $g$ in all three examples below. Setting $\sigma=0$ recovers the deterministic closed-loop system and allows a direct comparison between the deterministic and stochastic runs. All coefficients satisfy Assumptions~\ref{ass:standing} with constants as recorded in Remark~\ref{rem:verification}. All experiments reported in this section are reproducible with the archived
code~\cite{vonderHeydtRodrigues2026code}.

It remains to specify the subspaces $\mathcal{U}_L$ and $\widetilde{\mathcal{U}}_L$ in~\eqref{eq:KL-def}. The actuator subspace $\mathcal{U}_L = \operatorname{span}\{\1_{\clo_j}\}_{j=1}^N \subset H$ is spanned by indicator functions. The auxiliary subspace $\widetilde{\mathcal{U}}_L \subset V$ is spanned by sin-smoothed bump functions $\widetilde\phi_j$ supported on $\clo_j$. Following~\cite[Sect.~5, Eq.~(5.2)]{KunRodWal21}, we set
\[
  \widetilde\phi_j(x) \;=\; \sin\Bigl( \pi\,\tfrac{x_1 - p^j_1}{2r_1} \Bigr) \sin\Bigl( \pi\,\tfrac{x_2 - p^j_2}{2r_2} \Bigr) \1_{\clo_j}(x),
\]
where $p^j = (p^j_1, p^j_2)$ is the lower-left corner of $\clo_j$ and $r_1, r_2 > 0$ are its half-widths. The direct-sum condition~$\mathcal{U}_L\oplus\widetilde{\mathcal{U}}_L^\perp$ reduces to verifying that the $N \times N$ Gram matrix $\widetilde{B}^\top M B$ is invertible, which holds for any mesh fine enough,  because the~$\clo_j$ are pairwise disjoint.
%

Throughout this section we fix $\hat J = J_{x_1}=J_{x_2}=24$ ($|\mathcal J|=576$ modes), a $Q_1$ mesh of $n_{x_1}\times n_{x_2}=80\times 80$ elements, and a uniform time step $\tau=10^{-2}$. To assess this choice, we carry out a KL truncation study for $\hat J\in\{2,4,6,8,12,16,24,32\}$. For these square truncations, we write
\[
    \mathcal J_{\hat J}:=\{0,\ldots,\hat J-1\}^2.
\]
For the noise process we compare the reference error
\begin{equation}\label{eq:noise-error}
    \mathbb E\bigl\|W^{\mathcal J_{\hat J_\infty}}(T)-W^{\mathcal J_{\hat J}}(T)\bigr\|_H^2,
    \qquad \hat J_\infty=512,
\end{equation}
with the upper bound from Proposition~\ref{prop:noise-trunc-error}. In addition, we compute the Monte Carlo state error
\begin{equation}\label{eq:state-error}
    \mathbb E\bigl\|X^{\mathcal J_{\hat J_{\mathrm{ref}}}}(T)-X^{\mathcal J_{\hat J}}(T)\bigr\|_M^2,
    \qquad \hat J_{\mathrm{ref}}=40,
\end{equation}
with 50 samples for an uncontrolled and naturally stable test equation obtained from~\eqref{eq:ex1} by replacing $a(t)=-c_0X$ with $a(t)=+c_0X$ and setting $K\equiv 0$.

Figure~\ref{fig:kl-convergence} shows that the reference noise error decays consistently with the rate predicted by Proposition~\ref{prop:noise-trunc-error}. The same figure also indicates that, in the reported experiment, the state error decays substantially faster than the noise error. In view of Remark~\ref{rem:convergence-informal}, this faster decay should be understood as a numerical observation rather than as a consequence of a rigorous error estimate for the full three-layer scheme. At the analytical level, Proposition~\ref{prop:noise-trunc-error} justifies the approximation of the stochastic forcing, while the behavior of the state error is assessed here only through the simulations.

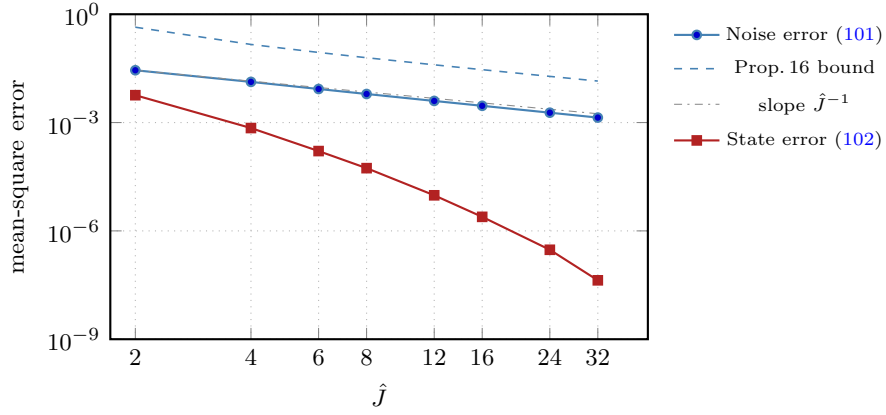
\begin{figure}[htbp]
    \centering
    \begin{tikzpicture}
\begin{axis}[xlabel={$\hat{J}$}, ylabel={mean-square error}, xmode={log}, ymode={log}, xmin={1.7214159528501156}, xmax={43.195481842432095}, ymin={1.0e-9}, ymax={1.0}, grid={major}, grid style={{dotted, gray!60}}, width={0.68\textwidth}, height={0.46\textwidth}, legend pos={outer north east}, legend style={{font=\footnotesize, draw=none, fill=none, row sep=0.5ex}}, tick label style={{font=\small}}, label style={{font=\small}}, line width={0.8pt}, minor tick style={{gray!50}}, xtick={{2,4,6,8,12,16,24,32}}, xticklabels={{2,4,6,8,12,16,24,32}}]
    \addplot+[color={myblue}, solid, thick, mark={*}, mark size={1.6pt}]
        table[row sep={\\}]
        {
            \\
            2.0  0.028099500497754754  \\
            4.0  0.013376304019183202  \\
            6.0  0.008508096459665632  \\
            8.0  0.006189401988916952  \\
            12.0  0.003974372286261595  \\
            16.0  0.002910606460987779  \\
            24.0  0.0018780330924629274  \\
            32.0  0.001373814315507843  \\
        }
        ;
    \addlegendentry {Noise error~\eqref{eq:noise-error}}
    \addplot+[color={myblue}, dashed, semithick, no marks]
        table[row sep={\\}]
        {
            \\
            2.0  0.43753626900255055  \\
            4.0  0.14584542300085018  \\
            6.0  0.08750725380051011  \\
            8.0  0.06250518128607864  \\
            12.0  0.03977602445477733  \\
            16.0  0.029169084600170037  \\
            24.0  0.019023316043589155  \\
            32.0  0.014114073193630663  \\
        }
        ;
    \addlegendentry {$\text{Prop.\,16 bound}$}
    \addplot+[color={mygray}, dash dot, thin, no marks]
        table[row sep={\\}]
        {
            \\
            2.0  0.028099500497754754  \\
            2.0714383899485265  0.027130423607195  \\
            2.1454285016762715  0.026194767596123554  \\
            2.2220614806309875  0.025291379867469266  \\
            2.3014317279024454  0.024419147574162456  \\
            2.3836370165113485  0.023576996248263267  \\
            2.4687786118519885  0.02276388847736778  \\
            2.5569613964370204  0.02197882262666132  \\
            2.6482939990980183  0.02122083160504472  \\
            2.7428889288009715  0.020488981673813668  \\
            2.8408627132415614  0.019782371296423452  \\
            2.9423360423909513  0.01910013002792231  \\
            3.0474339171689153  0.01844141744268526  \\
            3.1562858034274544  0.017805422099127475  \\
            3.269025791434579  0.017191360540122  \\
            3.3857927610547254  0.016598476327890393  \\
            3.5067305528292874  0.016026039112177363  \\
            3.631988145168003  0.015473343730561636  \\
            3.7617198378694714  0.014939709339794688  \\
            3.8960854421968842  0.014424478577097272  \\
            4.035250477743102  0.0139270167503806  \\
            4.179386376327598  0.013446711056394656  \\
            4.328670693176423  0.012982969825840483  \\
            4.48328732564537  0.012535221794516516  \\
            4.64342673975574  0.01210291539960115  \\
            4.809286204821783  0.011685518100204658  \\
            4.981070036458846  0.011282515721353446  \\
            5.15898984827158  0.010893411820598541  \\
            5.343264812532237  0.010517727076468092  \\
            5.534121930170195  0.010154998698010467  \\
            5.731796310405289  0.009804779854700688  \\
            5.9365314603694195  0.009466639126007823  \\
            6.148579585073201  0.00914015996994539  \\
            6.368201898087205  0.008824940209949972  \\
            6.595668943320455  0.008520591539456081  \\
            6.83126092829261  0.008226739043556891  \\
            7.07526806931036  0.0079430207371616  \\
            7.327990948973237  0.007669087119080551  \\
            7.589740886449245  0.007404600741488743  \\
            7.860840320976464  0.007149235794237395  \\
            8.141623209062972  0.006902677703501525  \\
            8.432435435874478  0.006664622744269069  \\
            8.733635241316364  0.006434777666194244  \\
            9.04559366133504  0.006212859332354211  \\
            9.368694984982223  0.005998594370464089  \\
            9.703337227805207  0.005791718836120584  \\
            10.049932622146208  0.0055919778876595055  \\
            10.408908124954854  0.005399125472226536  \\
            10.78070594373931  0.005212924022674601  \\
            11.165784081303935  0.0050331441649144465  \\
            11.56461689994455  0.004859564435357904  \\
            11.97769570579633  0.004691971008105782  \\
            12.40552935405406  0.0045301574315443445  \\
            12.848644875810464  0.004373924374025832  \\
            13.307588127284609  0.004223079378319836  \\
            13.782924462240274  0.004077436624532943  \\
            14.275239428422577  0.003936816701204677  \\
            14.785139488870694  0.0038010463842977275  \\
            15.31325276899534  0.0036699584238102123  \\
            15.860229830341263  0.0035433913377471075  \\
            16.42674447198785  0.003421189213197075  \\
            17.01349456057518  0.0033032015142695987  \\
            17.621202889977933  0.0031892828966558645  \\
            18.250618071686105  0.003079293028584949  \\
            18.902515456989466  0.002973096417954741  \\
            19.577698092101702  0.002870562245424658  \\
            20.276997707400746  0.0027715642032645625  \\
            21.001275742004072  0.0026759803397613333  \\
            21.751424404940984  0.002583692908991446  \\
            22.52836777422902  0.002494588225774505  \\
            23.333062935208606  0.002408556525629029  \\
            24.16650115953808  0.0023254918295579903  \\
            25.02970912630137  0.002245291813497555  \\
            25.92375018673283  0.002167857682268164  \\
            26.849725674116833  0.002093094047872728  \\
            27.808776260476083  0.0020209088119919802  \\
            28.802083361719696  0.001951213052532253  \\
            29.830870592981945  0.0018839209140859257  \\
            30.896405275944673  0.0018189495021695918  \\
            32.0  0.001756218781109672  \\
        }
        ;
    \addlegendentry {$\text{slope}\ \hat{J}^{-1}$}
    \addplot+[color={myred}, solid, thick, mark={square*}, mark size={1.6pt}, error bars/y dir={both}, error bars/y explicit={true}]
        table[row sep={\\}]
        {
            x  y  {y error plus}  {y error minus}  \\
            2.0  0.005727819390425388  0.0037411055504402097  0.0037411055504402097  \\
            4.0  0.0007055472636166125  0.0002526298375064279  0.00025262983750642796  \\
            6.0  0.00016244443776601664  4.219292116012892e-5  4.219292116012892e-5  \\
            8.0  5.4650639300063524e-5  1.3195498067398636e-5  1.3195498067398629e-5  \\
            12.0  9.676188647569752e-6  1.936995845492894e-6  1.936995845492894e-6  \\
            16.0  2.4486985518925744e-6  3.3556481838522156e-7  3.3556481838522156e-7  \\
            24.0  2.999939162408281e-7  3.662081274511397e-8  3.662081274511397e-8  \\
            32.0  4.278825832005637e-8  4.67584954722945e-9  4.67584954722945e-9  \\
        }
        ;
    \addlegendentry {State error~\eqref{eq:state-error}}
\end{axis}
\end{tikzpicture}
    \caption{KL truncation study for $\alpha=1.5$ and $\ell=0.25$.
    Blue solid: noise error $\mathbb E\bigl\|W^{\mathcal J_{\hat J_\infty}}(T)-W^{\mathcal J_{\hat J}}(T)\bigr\|_H^2$.
    Blue dashed: the bound from Proposition~\ref{prop:noise-trunc-error}.
    Gray dash-dotted: reference slope $\hat J^{-1}$.
    Red solid: Monte Carlo state error $\mathbb E\bigl\|X^{\mathcal J_{\hat J_{\mathrm{ref}}}}(T)-X^{\mathcal J_{\hat J}}(T)\bigr\|_M^2$ using 50 samples.}
    \label{fig:kl-convergence}
\end{figure}

Figure~\ref{fig:kl-convergence-sensitivity} shows the dependence of the same quantity on the covariance exponent $\alpha$ at fixed $\ell=0.25$. As $\alpha$ increases, the decay in $\hat J$ becomes steeper, in agreement with the rate $(\hat J-1)^{-(2\alpha-2)}$ from Proposition~\ref{prop:noise-trunc-error}. On the basis of Figures~\ref{fig:kl-convergence} and~\ref{fig:kl-convergence-sensitivity}, the choice $J_{x_1}=J_{x_2}=24$ is sufficient for the simulations below.

\begin{figure}[htbp]
    \centering
    \begin{tikzpicture}
\begin{axis}[xmode={log}, ymode={log}, xmin={1.7214159528501156}, xmax={43.195481842432095}, ymin={1.0e-11}, ymax={100.0}, grid={major}, grid style={{dotted, gray!60}}, tick label style={{font=\small}}, label style={{font=\small}}, title style={{font=\normalsize}}, line width={0.8pt}, minor tick style={{gray!50}}, xtick={{2,4,6,8,12,16,24,32}}, xticklabels={{2,4,6,8,12,16,24,32}}, width={0.62\textwidth}, height={0.42\textwidth}, legend image code/.code={{\draw[#1, solid, line width=1.1pt] (0cm,0cm) -- (0.28cm,0cm);}}, xlabel={$\hat{J}$}, ylabel={mean-square error}, legend pos={outer north east}, legend columns={1}, legend cell align={{left}}, legend style={{font=\scriptsize, draw=none, fill=none, row sep=0.4ex}}]
    \addplot+[color={myblue}, solid, thick, mark={*}, mark size={1.5pt}]
        table[row sep={\\}]
        {
            \\
            2.0  0.13027978017864156  \\
            4.0  0.08625853954285569  \\
            6.0  0.06702702951747372  \\
            8.0  0.05602988736459155  \\
            12.0  0.043485642467705476  \\
            16.0  0.03625423983491624  \\
            24.0  0.027889163944883277  \\
            32.0  0.022998958964000615  \\
        }
        ;
    \addlegendentry {$\alpha=1.25$}
    \addplot+[color={myblue}, dashed, semithick, no marks, forget plot]
        table[row sep={\\}]
        {
            \\
            2.0  1.7975835951399404  \\
            4.0  1.0378353725448997  \\
            6.0  0.8039038227942734  \\
            8.0  0.6794227362270997  \\
            12.0  0.5419918467615737  \\
            16.0  0.46413408849284304  \\
            24.0  0.3748220900592146  \\
            32.0  0.32285554452348486  \\
        }
        ;
    \addplot+[color={teal!70!black}, solid, thick, mark={*}, mark size={1.5pt}]
        table[row sep={\\}]
        {
            \\
            2.0  0.028099500497754754  \\
            4.0  0.013376304019183202  \\
            6.0  0.008508096459665632  \\
            8.0  0.006189401988916952  \\
            12.0  0.003974372286261595  \\
            16.0  0.002910606460987779  \\
            24.0  0.0018780330924629274  \\
            32.0  0.001373814315507843  \\
        }
        ;
    \addlegendentry {$\alpha=1.50$}
    \addplot+[color={teal!70!black}, dashed, semithick, no marks, forget plot]
        table[row sep={\\}]
        {
            \\
            2.0  0.43753626900255055  \\
            4.0  0.14584542300085018  \\
            6.0  0.08750725380051011  \\
            8.0  0.06250518128607864  \\
            12.0  0.03977602445477733  \\
            16.0  0.029169084600170037  \\
            24.0  0.019023316043589155  \\
            32.0  0.014114073193630663  \\
        }
        ;
    \addplot+[color={orange!85!black}, solid, thick, mark={*}, mark size={1.5pt}]
        table[row sep={\\}]
        {
            \\
            2.0  0.002228143575774469  \\
            4.0  0.0005423445042344356  \\
            6.0  0.0002259142416874485  \\
            8.0  0.00012171015940680256  \\
            12.0  5.1458442277402955e-5  \\
            16.0  2.8159506762643824e-5  \\
            24.0  1.214895226787048e-5  \\
            32.0  6.722366418904571e-6  \\
        }
        ;
    \addlegendentry {$\alpha=2.00$}
    \addplot+[color={orange!85!black}, dashed, semithick, no marks, forget plot]
        table[row sep={\\}]
        {
            \\
            2.0  0.05750457999095845  \\
            4.0  0.00638939777677316  \\
            6.0  0.0023001831996383383  \\
            8.0  0.0011735628569583356  \\
            12.0  0.0004752444627351938  \\
            16.0  0.00025557591107092646  \\
            24.0  0.0001087043100018118  \\
            32.0  5.9838272623265805e-5  \\
        }
        ;
    \addplot+[color={purple!70!black}, solid, thick, mark={*}, mark size={1.5pt}]
        table[row sep={\\}]
        {
            \\
            2.0  2.563139872420201e-5  \\
            4.0  1.7187699109757928e-6  \\
            6.0  3.110034669273382e-7  \\
            8.0  9.202606782290684e-8  \\
            12.0  1.6748630512976488e-8  \\
            16.0  5.05992140857137e-9  \\
            24.0  9.511695509998794e-10  \\
            32.0  2.9324109758260483e-10  \\
        }
        ;
    \addlegendentry {$\alpha=3.00$}
    \addplot+[color={purple!70!black}, dashed, semithick, no marks, forget plot]
        table[row sep={\\}]
        {
            \\
            2.0  0.002340936262071654  \\
            4.0  2.890044767989696e-5  \\
            6.0  3.7454980193146466e-6  \\
            8.0  9.749838659190562e-7  \\
            12.0  1.598890965146953e-7  \\
            16.0  4.624071628783514e-8  \\
            24.0  8.36523690978682e-9  \\
            32.0  2.5347948363617653e-9  \\
        }
        ;
\end{axis}
\end{tikzpicture}
    \caption{Dependence of the KL truncation error on $\alpha$ at fixed $\ell=0.25$.
    For each color, the solid curve shows noise error $\mathbb E\bigl\|W^{\mathcal J_{\hat J_\infty}}(T)-W^{\mathcal J_{\hat J}}(T)\bigr\|_H^2$ and the dashed curve shows the corresponding bound from Proposition~\ref{prop:noise-trunc-error}. Larger values of $\alpha$ yield faster decay in $\hat J$.}
    \label{fig:kl-convergence-sensitivity}
\end{figure}
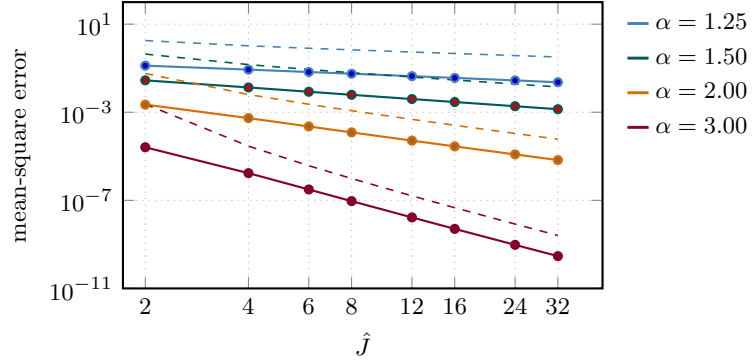

\subsection{Rotating convection-diffusion: additive versus multiplicative noise}
\label{subsec:sim-linear}

As a first example we consider a passive scalar transported by a time-periodic rotating background flow, a standard model in fluid mixing and geophysical transport. The closed-loop system reads
\begin{subequations}\label{eq:ex1}
\begin{align}
    \mathrm d X  + \Bigl( -\nu\Delta X &+ a(t)X + b(t)\cdot \nabla X \Bigr) \, \mathrm d t  \notag
    \\&= \sigma\bigl(\xi(x) + z\,X\bigr)\,\mathrm d W_t + \sum_{i=1}^{N} [K_L^{[\lambda]}X]_i \, \1_{\clo_i}(x) \, \mathrm d t,
    \label{eq:ex1-pde}\\
    &\partial_n X\rest{\partial\mathcal O} = 0,\qquad X\rest{t=0} = X_0,
    \label{eq:ex1-bic}
\end{align}
\end{subequations}
where we set $a(t) = -c_0$ and $b(t) = \beta(-\sin(\omega t), \cos(\omega t))^\top$ with $c_0,\beta,\omega,z\geq 0$.
The instability of the free dynamic (uncontrolled system) is driven by the constant reaction coefficient $a(t) = -c_0 < 0$. 

\begin{table}[htbp]
\centering
    \caption{Parameters for Example~\ref{subsec:sim-linear}.}
    \label{tab:params-ex1}
    \begin{tabular}{cccccccccc}
    \toprule
    $\nu$ & $L_{x_1}, L_{x_2}$ & $c_0$ & $\beta$ & $\omega$ & $\sigma$ & $N$ & $\lambda$ & $T$ & $X_0(x)$\\
    \midrule
    $0.1$ & $1$ & $2.5$ & $1.0$ & $\pi/2$ & $5$ & $9$ & $0.5$ & $2.5$ & $\sin(\pi x_1) \sin(\pi x_2)$\\
    \bottomrule
\end{tabular}
\end{table}

The feedback uses $N=9$ ($L =3$) actuators arranged on a uniform $3\times3$ grid. The union of the actuator supports occupies $25\%$ of the spatial domain, and the feedback strength is $\lambda=0.5$. The noise amplitude $\sigma = 5$ is shared by both sub-cases below, allowing a direct comparison between the two noise regimes.

We consider two noise regimes. In sub-case~(a) we set $z = 1$ and
$\xi \equiv 0$, so the noise is purely multiplicative and Theorem~\ref{thm:main}~(ii) predicts almost-sure pathwise decay to zero. In sub-case~(b) we set $z = 0$ and $\xi(x) = \sin(\pi x_1)\sin(\pi x_2)$, so the noise is purely additive and Theorem~\ref{thm:main}~(i) predicts mean-square stabilization to a finite noise floor. The results are shown in Figure~\ref{fig:ex1}.

In sub-case~(a), Figure~\ref{fig:ex1-det-mult} displays the evolution of the $H$-energy for four configurations: uncontrolled and stabilized dynamics in the deterministic case $(\sigma=0)$, and the corresponding uncontrolled and stabilized dynamics under multiplicative noise. The observed behavior is consistent with the theoretical findings. In both deterministic and stochastic settings, the stabilized energy decays, and the stochastic closed-loop curve closely follows its deterministic counterpart, illustrating that stabilization persists under multiplicative noise with amplitude $\sigma=5$.

In sub-case~(b), Figure~\ref{fig:ex1-additive} displays the evolution of the mean $H$-energy for the different configurations. In the stochastic closed-loop case, this quantity is approximated by averaging $\|X(t)\|_H^2$ over $N_{\mathrm{MC}}=20$ independent Monte Carlo realizations. The stabilized mean energy decays rapidly and then approaches a stationary level, whereas the uncontrolled dynamics grow over the simulated time interval. The empirical Monte Carlo mean approaches the asymptotic noise floor $\|\xi\|_H^2\Tr(Q)/\lambda^2$, while the upper bound from Theorem~\ref{thm:main} remains above this level by only a moderate factor.

\begin{figure}[htbp]
    \centering
    \subfigure[Rotating convection-diffusion: deterministic vs.\ multiplicative noise]{
        \input{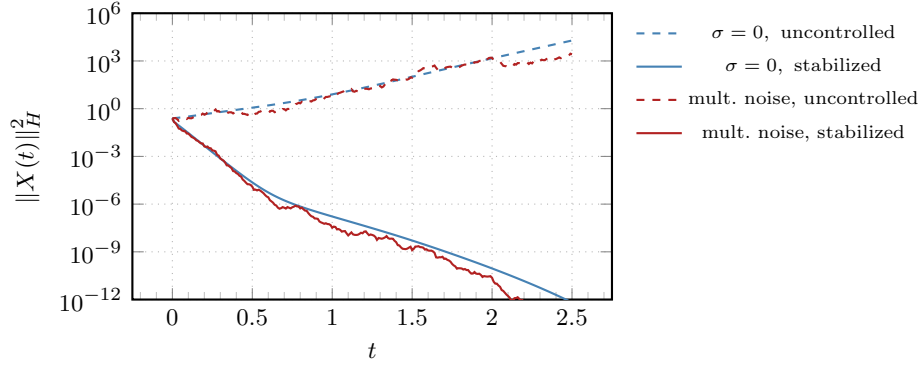}
        \label{fig:ex1-det-mult}
    }
    \hfill
    \subfigure[Additive noise: MC bundle and noise floor]{
        \input{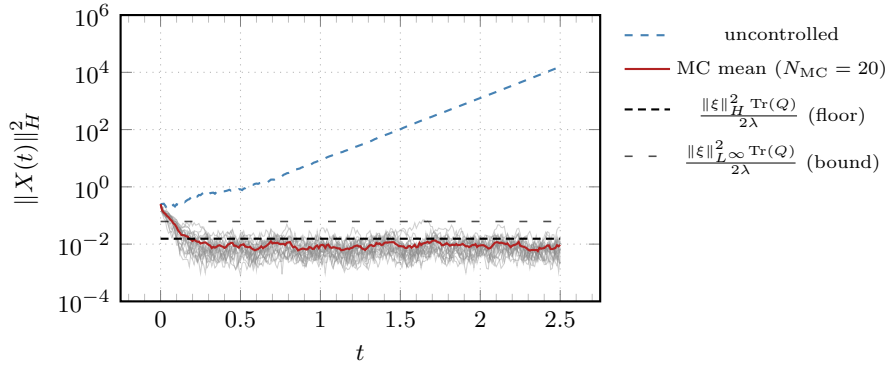}
        \label{fig:ex1-additive}
    }
    \caption{Example~\ref{subsec:sim-linear}. Evolution of the $H$-energy and the mean $H$-energy for different configurations. \emph{Top}: deterministic and multiplicative-noise cases. \emph{Bottom}: additive-noise case with $N_{\mathrm{MC}}=20$ Monte Carlo realizations, empirical mean, asymptotic noise floor, and theorem bound.}
    \label{fig:ex1}
\end{figure}

To demonstrate the effectiveness of the feedback, we consider a numerical test in which the feedback is switched on and off. More precisely, we multiply the control term in~\eqref{eq:ex1} by $\1_I(t)$, where $I\subset[0,T]$ denotes the time interval on which the actuators are active. Both runs use pure multiplicative noise $(z=1,\xi\equiv0)$, with the same parameters and noise realization as in Example~\ref{subsec:sim-linear}.

In Run~A, the actuators are active on $[0,T_{\mathrm{off}}]$ and are then switched off permanently. As shown in Figure~\ref{fig:ex5-intermittent}, the energy decreases during the controlled phase, but starts to grow again after the feedback is removed. In Run~B, the actuators are reactivated at $t=T_{\mathrm{on}}$. From that time onward, the energy again decays at a rate comparable to that observed during the initial controlled phase. This indicates that the feedback $K_L^{[\lambda]}$ can restabilize the system from an elevated energy level. 

\begin{figure}[htbp]
    \centering
    \input{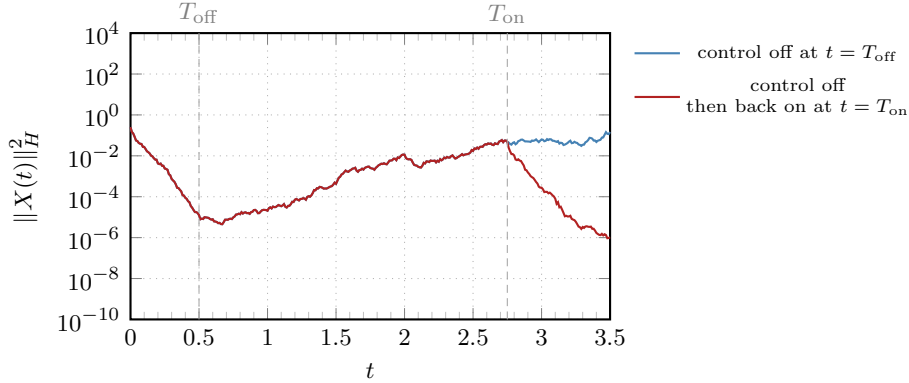}
    \caption{Intermittent control under multiplicative noise
$(z=1,\xi\equiv0)$. Run~A: the actuators are switched off after
$T_{\mathrm{off}}$, and the energy starts to grow again. Run~B: the
actuators are reactivated at $T_{\mathrm{on}}$, after which the feedback restabilizes the system.}
  \label{fig:ex5-intermittent}
\end{figure}

\subsection{Reaction-diffusion with saturating nonlinearity: effect of actuator count}
\label{subsec:sim-nonlinear}

As a second example we consider a reaction-diffusion equation with a saturating nonlinear source, modelling autocatalytic growth with a bounded rate. The closed-loop system reads
\begin{subequations}\label{eq:ex2}
\begin{align}
    \mathrm d X
    + \Bigl(-\nu\Delta X - \gamma\,\arctan(X)\Bigr)\,\mathrm d t
    &= \sigma\,z\,X\,\mathrm d W_t
    + \sum_{i=1}^{N}[K_L^{[\lambda]}X]_i\,\1_{\clo_i}(x)\,\mathrm d t,
    \label{eq:ex2-pde}\\
    \partial_n X\rest{\partial\mathcal O} = 0,\qquad&
    X\rest{t=0} = X_0,
    \label{eq:ex2-bic}
\end{align}
\end{subequations}
with $\gamma, z\geq 0$.  There is no convection term; the only source of instability is the nonlinear reaction $f(u)=-\gamma\arctan(u)$, which contributes as a linear reaction with rate $\gamma$ at the origin. The noise is purely multiplicative with coefficient $z > 0$, placing this example in the almost-sure pathwise regime of Theorem~\ref{thm:main}~(ii).

\begin{table}[ht]
\centering
    \caption{Parameters for Example~\ref{subsec:sim-nonlinear}.}
    \label{tab:params-ex2}
    \begin{tabular}{ccccccccc}
    \toprule
    $\nu$ & $L_{x_1}, L_{x_2}$ & $\gamma$ & $z$ & $\sigma$ & $N$ & $\lambda$ & $T$ & $X_0(x)$ \\
    \midrule
    $0.1$ & $1$ & $5$ & $1.0$ & $5$ & $4$ or $9$ & $0.5$ & $3$ & $\sin(\pi x_1) \sin(\pi x_2)$ \\
    \bottomrule
\end{tabular}
\end{table}

The feedback is constructed with total support measure equal to $25\%$ of the whole domain and feedback strength $\lambda = 0.5$. We compare two actuator configurations, namely $N=4$ and $N=9$, as shown in Figure~\ref{fig.suppActR}. In all simulations, the noise amplitude is fixed at $\sigma=5$.

Figure~\ref{fig:ex2a} depicts the evolution of the energy $\|X(t)\|_H^2$ for all four runs. The uncontrolled trajectory grows exponentially, in agreement with the instability induced by the reaction rate $\gamma$. For $N=9$, the deterministic closed-loop solution decays exponentially down to machine precision. The corresponding stochastic closed-loop trajectory closely follows the deterministic one, illustrating pathwise decay at the prescribed rate despite the multiplicative noise with amplitude $\sigma=5$. By contrast, the stochastic run with $N=4$ is not stabilized: the energy decreases initially but then levels off, indicating that this smaller number of actuators is not sufficient to stabilize the system for the present value of $\gamma$.

Figure~\ref{figex2b} depicts snapshots of the $N=9$ stabilized state $X(x,t)$ at eight different time instances, arranged in a $2\times4$ grid with time increasing from left to right and from top to bottom. Panel~1, corresponding to $t=0$, shows the initial condition $X_0(x)=\sin(\pi x_1)\sin(\pi x_2)$, centered at $(0.5,0.5)$ with maximum value one. Panels~2--4 $(t=0.06,0.09,0.14)$ show the early transient, during which the effect of the feedback is most visible. The nine red circular regions mark the actuator supports $\clo_i$, arranged on a uniform $3\times3$ grid. The spatial imprint of these actuator regions is visible against the decaying state profile. Panels~5--7 $(t=0.21,0.31,0.47)$ show the late transient, where this imprint fades as the state amplitude continues to decrease. In Panel~8 $(t=0.71)$, the solution is essentially zero on the common color scale, consistently with the energy crossing the threshold $\varepsilon=10^{-5}$ at approximately $t=0.71$. All eight panels share a symmetric logarithmic color scale with crossover threshold $\varepsilon_{\rm c}=10^{-4}$. This avoids the singularity at zero, resolves the five-decade decay range, and preserves the sign of $X$. The colorbar ticks correspond to $X\in\{-1,-10^{-1},-10^{-2},0,10^{-2},10^{-1},1\}$.

\begin{figure}[htbp]
    \centering
    \input{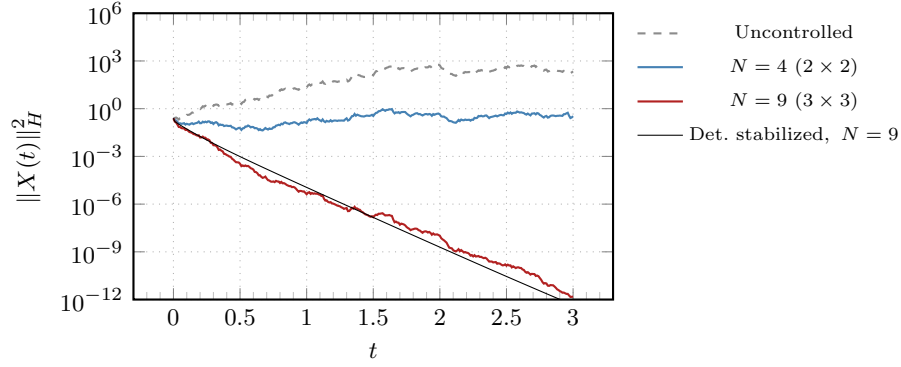}
    \caption{Example~\ref{subsec:sim-nonlinear}. Energy evolution $\|X(t)\|_H^2$ for the uncontrolled, stochastic stabilized, and deterministic reference systems. For the chosen nonlinearity strength $\gamma$, stabilization is achieved with $N=9$ actuators but not with $N=4$.}
    \label{fig:ex2a}
\end{figure}

\begin{figure}[htbp]
    \centering
    \includegraphics[width=\textwidth]{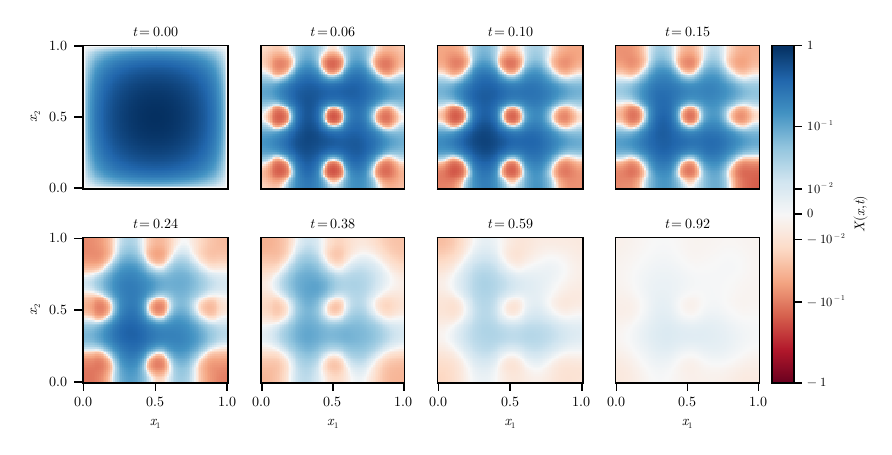}
    \caption{Example~\ref{subsec:sim-nonlinear}. Snapshots of the stabilized state for the $N=9$ actuator configuration, shown at eight time instances. Time increases from left to right and from top to bottom.}
    \label{figex2b}
\end{figure}

\backmatter

%
%

%

\section*{Conclusion}

We established a finite-dimensional feedback stabilization framework for nonautonomous stochastic semilinear parabolic-type equations driven by Wiener noise with trace-class covariance. The feedback is constructed from localized actuator functions, taken here as characteristic functions of small subdomains, and an oblique-projection mechanism. Under suitable conditions, the resulting closed-loop system is mean-square exponentially stable; in the purely multiplicative case, pathwise exponential stability is also obtained.

The analysis shows that, for a given actuator construction, the number of actuators has to be chosen in relation to the coefficients of the equation, including the diffusion parameter, lower-order terms, and nonlinear contributions. The numerical experiments support this observation: for the same actuator layout, the $N=4$ configuration is not sufficient in the tested parameter regime, whereas the $N=9$ configuration yields the predicted decay. The intermittent-control experiment further illustrates the effect of the feedback, since switching the control off leads to renewed growth of the energy, while reactivation restores decay.

The present work also suggests several directions for future research. These include more general nonlinearities, in particular locally Lipschitz continuous terms or L\'evy noise. Another natural direction is the extension of the feedback strategy to other classes of stochastic differential equations, such as stochastic damped wave equations. A systematic study of actuator placement and intermittent feedback mechanisms is left for future
work.

\section*{Declarations}

\noindent\textbf{Funding}. The work of J.H. was supported by the German Research Foundation (DFG) through the Collaborative Research Centre SFB~1432, ``Fluctuations and Nonlinearities in Classical and Quantum Matter beyond Equilibrium''.\\
The work of S.R. is funded by portuguese national funds through the FCT -- Funda\c{c}\~{a}o para a Ci\^{e}ncia e a Tecnologia, I.P., under the scope of the projects UID/00297/2025 (\url{https://doi.org/10.54499/UID/00297/2025}) and UID/PRR/00297/2025 (\url{https://doi.org/10.54499/UID/PRR/00297/2025}) (Center for Mathematics and Applications -- NOVA Math)

\noindent\textbf{Competing interests}. The authors declare that they have no competing interests.

\noindent\textbf{Ethics approval and consent to participate}. Not applicable.

\noindent\textbf{Consent for publication}. All authors have approved the manuscript and consent to its publication.

\noindent\textbf{Data availability}. The data generated and analysed during the numerical simulations of this study are available from the corresponding author upon reasonable request.

\noindent\textbf{Materials availability}. Not applicable.

\noindent\textbf{Code availability.} The code used for the numerical simulations, including the finite-element assembly, the feedback-stabilized backward Euler-Maruyama solver, and the scripts producing all figures, is archived on Zenodo under the MIT license at \url{https://doi.org/10.5281/zenodo.21220182},~\cite{vonderHeydtRodrigues2026code}.

\noindent\textbf{Author contribution}. All authors contributed equally to the theoretical analysis, the numerical methodology, the implementation and interpretation of the numerical experiments, and the writing of the manuscript. All authors read and approved the final manuscript.


\bibliography{SPDE}

\end{document}